\documentclass{article}

\newcommand{\Span}{\mathrm{Span}}
\newcommand{\R}{\mathbb{R}}
\newcommand{\C}{\mathbb{C}}
\newcommand{\V}{\mathbb{V}}
\newcommand{\Or}{\mathrm{O}}
\newcommand{\U}{\mathrm{U}}
\newcommand{\Sp}{\mathrm{Sp}}
\newcommand{\so}{\mathfrak{so}}
\newcommand{\spp}{\mathfrak{sp}}
\newcommand{\uu}{\mathfrak{u}}
\newcommand{\Aut}{\mathrm{Aut}}
\newcommand{\Der}{\mathrm{Der}}
\newcommand{\End}{\mathrm{End}}
\newcommand{\Vn}{\V_{2n+1}}

\newcommand{\vprod}{\mathop{\veebar}\nolimits}
\newcommand{\vprodD}[1]{\mathop{\veebar}\nolimits_{#1}}
\newcommand{\vm}{\vprod}
\newcommand{\vmT}[1]{\vprodD{#1}}

\newcommand{\vijo}{\mathop{\circ_{\mathfrak v}}\nolimits}
\newcommand{\VJ}{\mathcal{V}\!\mathcal{J}} 
\usepackage[T1]{fontenc}
\usepackage[utf8]{inputenc}
\usepackage{   amsthm,
   amsmath,
   amssymb,
   enumitem,
   geometry,
   caption,
   graphicx,
   hyperref,
  }
\usepackage{amsfonts, mathdesign}

\theoremstyle{plain} 
\newtheorem{theorem}{Theorem}[section] 
\newtheorem{corollary}[theorem]{Corollary}
\newtheorem{proposition}[theorem]{Proposition}
\newtheorem{lemma}[theorem]{Lemma} 
\theoremstyle{definition} 
\newtheorem{example}[theorem]{Example} \newtheorem{remark}[theorem]{Remark}
\newtheorem{definition}[theorem]{Definition}

\title{Structure of Vidinli Algebras: Simple Non-associative Lie- and Jordan-Admissible Algebras}

\author{%
  Olcay CO\c{S}KUN\thanks{\textup{Bo\u{g}azi\c{c}i University}; \texttt{olcay.coskun@bogazici.edu.tr}}
  \and
  Alp EDEN\thanks{\.Izmir, T\"urkiye; \texttt{alp.eden5@gmail.com}}%
}
\date{}

\begin{document}
\maketitle

\begin{abstract}
We extend the three-dimensional Vidinli algebra of Vidinli H\"useyin Tevfik Pasha to a family of unital real algebras \( \Vn \), one in each odd dimension, built from the Euclidean inner product and
a symplectic form on \( e_1^\perp \). Our main result is a local-to-global characterization by copies of \( \C \) and \( \V_3 \). A unital real algebra on Euclidean \( \R^{2n+1} \) with unit \( e_1 \) is isometrically isomorphic to \( \Vn \) if and only if every principal \( 2 \)-plane is an isometric copy of \( \C \), every principal \( 3 \)-plane is a subalgebra, and through every direction of \( e_1^\perp \) there passes at least one, and only finitely many, principal \( 3 \)-planes isomorphic to \( \V_3 \). We show that \( \Vn \) is central simple and power-associative with automorphism group \( \U(n) \), that its symmetrization is a Jordan algebra of Clifford type, and that its commutator is the Heisenberg algebra \( \mathfrak{h}_n \). 
\smallskip



\end{abstract}

\section{Introduction}

In his 1882 monograph \emph{Linear Algebra}~\cite{Pacha}, reissued in a second edition ten years
later~\cite{Pacha2}, the Ottoman mathematician Vidinli Hüseyin Tevfik Pasha defined a bilinear
product on \( \R^3 \). His aim was a coordinate-free calculus for spatial vectors, free of the
quaternionic formalism. The resulting operation, now called the \emph{Vidinli multiplication}, is non-associative
and non-commutative, has zero divisors, and recovers the inner product, the cross product and the
scalar triple product from a single unital product~\cite{CEK, Eden}. Despite this structure it has
remained a historical curiosity. We show that it is the smallest member of an infinite, completely tractable family of simple
non-associative algebras \( \Vn \), one in each odd dimension, whose structure we determine
completely.

The starting point is a coordinate-free description of the three-dimensional product
(Proposition~\ref{prop:coordfree3}): for \( a,b\in\R^3 \),
\[
a\vm b = (a\cdot e_1)\,b + (b\cdot e_1)\,a - (a\cdot b)\,e_1 + \big(e_1\cdot(a\times b)\big)\,e_1 ,
\]
where \( e_1 \) is a fixed unit vector. The key to generalization is to recognize the last term as a
symplectic form \( \omega_0(u,v)=e_1\cdot(u\times v) \) on the plane \( V_0=e_1^\perp \). Replacing
\( \R^3 \) by \( \R^{2n+1} \) and \( \omega_0 \) by any symplectic form \( \omega \) on \( V_0 \), we
obtain, for every \( n\ge1 \), the \emph{Vidinli algebra}
\[
\Vn = (\R^{2n+1},\vm),\qquad
a\vm b = (a\cdot e_1)\,b + (b\cdot e_1)\,a - (a\cdot b)\,e_1 + \omega(a,b)\,e_1 ,
\]
unital with unit \( e_1 \) and built from nothing more than the inner product and the form \( \omega \).

The symmetric and skew parts realize two classical structures at once. The symmetrization is a Jordan
algebra, though not a formally real one. The commutator \( [a,b]=2\,\omega(a,b)\,e_1 \) is central,
making \( \R^{2n+1} \) the Heisenberg Lie algebra \( \mathfrak h_n \) with center \( \R e_1 \). Thus
\( \Vn \) is both \emph{Jordan-admissible} and \emph{Lie-admissible}, while the product itself is
neither commutative, associative, alternative nor flexible. As Elduque and
R\'andez-Ib\'a\~nez~\cite{ElduqueRandez} observe, it is intrinsically a conic algebra all of whose
commutators are central. A first structural consequence is that every \( \Vn \) is \emph{simple}. 

The basic invariants are determined in Section 4 as follows. The automorphism group is \( \Aut(\Vn)\cong\U(n) \), acting
on \( V_0\cong\C^n \) as the unitary group preserving both the inner product and \( \omega \), and the derivation algebra is \( \uu(n) \).
Moreover, \( \Vn \) is power-associative and central simple, so that
\( \mathrm M(\Vn)\cong M_{2n+1}(\R) \). The \( 2 \)-torsion subgroup \( \{\pm1\}^n \) of the diagonal torus of \( \Aut(\Vn) \) induces a fine
grading of \( \Vn \) by \( (\mathbb{Z}/2)^n \). The norm \( \|a\|^2 \) is positive definite and every nonzero element of \( \Vn \) has a two-sided
inverse; yet \( \Vn \) is not a division algebra, since for \( a\in V_0 \) both the left multiplication operator \( L_a \) and the right multiplication operator \( R_a \) are singular and \( a \) is a two-sided zero divisor. The two notions come apart as far as
they possibly can: for \( a\in V_0\setminus\{0\} \) the solutions of \( a\vm b=b\vm a=e_1 \) form an
affine subspace of dimension \( 2n-2 \). Without
associativity, the existence of two-sided inverses carries almost no information.

In this paper, we determine the subalgebras of \( \Vn \) and characterize the algebra by them.
We call a
\( k \)-dimensional subspace containing \( e_1 \) a \emph{principal \( k \)-plane}. Note that every
principal \( 2 \)-plane is a copy of \( \C \). Every principal
\( 3 \)-plane \( \Span\{e_1,u,v\} \), with \( u,v \) orthonormal, is a member \( \mathcal V_{|t| }\)
of a \emph{twisted family}, where \( t=\omega(u,v) \). This family runs from \( \V_3 \) at
\( |t|=1 \) to the Jordan algebra \( \VJ_3 \) at \( t=0 \).
Moreover, the maximal Jordan subalgebras correspond to the Lagrangian subspaces of \( V_0 \). These local data determine the algebra, and this leads us to the main
result of the paper.

\smallskip
\noindent\textbf{Theorem~\ref{thm:vidinli-characterization}.}
\emph{Let \( n\ge1 \) and let \( A=(\R^{2n+1},\ast) \) be a unital \( \R \)-algebra with unit
\( e_1 \), where \( \R^{2n+1} \) carries its standard inner product. Then \( A \) is isometrically
isomorphic to \( \V_{2n+1} \) if and only if every principal \( 2 \)-plane of \( A \) is an isometric
copy of \( \C \), every principal \( 3 \)-plane is a subalgebra, and through every direction of
\( V_0 \) there passes at least one, and only finitely many, principal \( 3 \)-planes isomorphic to
\( \V_3 \).}

\smallskip

Subsequent to the appearance of a preprint version of the present work~\cite{CEpreprint}, Elduque and
R\'andez-Ib\'a\~nez~\cite{ElduqueRandez}
introduced an intrinsic, field-independent notion of Vidinli algebra, namely a conic algebra all of
whose commutators are scalar. The algebras studied here are the real, positive-definite case. In
characteristic different from 2 such an algebra is governed by a single, not necessarily symmetric,
bilinear form, so that \( \Vn \) is the Vidinli algebra of the bilinear form
\( (a,b)\mapsto a\cdot b-\omega(a,b) \). Moreover, they extend the determination of the
automorphism group, derivation algebra, center and multiplication algebra to arbitrary fields, and
develop a separate characteristic 2 theory. The present work is complementary: over \( \R \) we obtain the
explicit unitary symmetry and the local-to-global characterization.

\section{Historical background}\label{sec:history}
In a linear algebra book first published in 1882 \cite{Pacha} a non-associative vector product was defined for space vectors. Its author, Hüseyin Tevfik Pasha, was a high ranking military officer in the Ottoman Army who spent an extensive period as a military attache in France, USA and Germany \cite{CEK}. Since he was born in Vidin, then in the Ottoman Empire and now in Bulgaria, Hüseyin Tevfik Pasha was also called Vidinli, meaning from Vidin. The product he defined was non-associative, non-commutative and admitted zero divisors. In his book the author tried to recover the geometric and algebraic applications of quaternion multiplication to space vectors without recourse to quaternions. The author's success lay mainly in the fact that the cross product, the inner product and the triple (box) product could all be recovered using the v-product \cite{Eden}. In the present work these geometrically describable operations arise from the commutator and anti-commutator of Vidinli's multiplication.

Hüseyin Tevfik Pasha was not the first to recognize the importance of the cross product and the dot product in the analysis of space vectors \cite{Smith}. It took the efforts of various mathematicians such as Gibbs and Heaviside to rework the theory of vector calculus using the cross product and dot product and leaving quaternion analysis behind. In fact, Gibbs' famous lecture notes on vector analysis do not mention quaternions at all \cite{Wilson}.

In the late 19th century various realizations of vector spaces were in competition and quaternions were one of the main contenders for the title \cite{Crowe}. Students and followers of Hamilton proposed linear associative algebras as a natural generalization of quaternions. Both Charles Sanders Peirce and James B. Shaw were aware of Tevfik Pasha's book, but neither ventured into the non-associative wilderness\footnote{The first edition of the book was sent to Charles Sanders Peirce, who suggested that it should be transferred to Harvard University's library. There is a digital copy of the second edition in the internet archives that belonged to the late James Shaw.}.

Hamilton communicated his findings on quaternions to his friend Graves, who remarked that ``If with your alchemy you can make three pounds of gold, why should you stop there?''. Graves's introduction of octonions was then overshadowed by the flurry of activity on quaternions \cite{Crowe}, \cite{Baez}. Once a three-dimensional vector product is introduced in terms of cross and dot products, it is natural to ask whether it has generalization to higher dimensions. An affirmative answer can be given in seven space dimensions where a cross product can be defined. In the present paper, we give an equivalent, coordinate-free definition of the v-product in three
dimensions, and we extend the results of \cite{CEK} to all odd dimensions using a skew-symmetric form.
Note that the classification of linear associative algebras was then available only in dimensions 3,
4, 5 and 6 \cite{Peirce}. The non-associative products of the period were thus introduced with very
little general theory to guide them.

\section{The Vidinli algebra \texorpdfstring{$\V_3$}{V3}}\label{sec:v3}

In this section, we collect the results on the three-dimensional Vidinli algebra that will be used
throughout the paper. We refer to \cite{CEK} for details and proofs. Note that Vidinli's original
definition in \cite{Pacha2} is geometric; we use the identification of \cite{CEK}.

Let \( \R^3 \) be the standard Euclidean space with the standard inner product and the standard cross product \( \times \), and an ordered orthonormal basis \( \{e_1, e_2, e_3\} \) satisfying \( e_2 \times e_3 = e_1 \).

\begin{definition}[Vidinli Multiplication]
\label{def:V3}
The \emph{Vidinli multiplication} on \( \R^3 \) is the bilinear map \( \vm\colon \R^3 \times \R^3 \to \R^3 \) defined by
\begin{equation}
\label{eq:V3-coord}
a \vm b = (a_1 b_1 - a_2 b_2 - a_3 b_3 + a_2 b_3 - a_3 b_2)\,e_1 + (a_1 b_2 + a_2 b_1)\,e_2 + (a_1 b_3 + a_3 b_1)\,e_3.
\end{equation}
The vector \( e_1 \) is called the \emph{principal direction}, and \( \V_3 = (\R^3, \vm) \) is the \emph{Vidinli algebra}.
\end{definition}

It is shown in \cite{CEK} that \( \V_3 \) is a unital algebra with identity \( e_1 \), and is neither commutative nor associative. The map \( \bar{a} = (a_1, -a_2, -a_3) \) is an involution satisfying \( \bar{a} \vm a = a \vm \bar{a} = \|a\|^2 e_1 \), so every \( a \notin e_1^\perp \) is invertible and \( a^{-1} = \bar{a}/\|a\|^2 \). The set of zero divisors is exactly the plane \( e_1^\perp \). For any unit vector \( u \in e_1^\perp \), the subspace \( \Span_\R\{e_1, u\} \) is a subalgebra isomorphic to \( \C \). Thus \( \V_3 \) is the union of a pencil of copies of \( \C \), all meeting along the principal axis
\( \R e_1 \).

The following result identifies the Vidinli multiplication in terms of the inner product and the cross product, uncovering its geometric content and decomposing it into symmetric and skew-symmetric parts.

\begin{proposition} \label{prop:coordfree3}
For all \( a, b \in \R^3 \),
\begin{equation}
\label{eq:V3-intrinsic-eq}
a \vm b = (a \cdot e_1)\,b + (b \cdot e_1)\,a - (a \cdot b)\,e_1 + (e_1 \cdot (a \times b))\,e_1.
\end{equation}
\end{proposition}

\begin{proof}
Write \( a = a_1 e_1 + a_2 e_2 + a_3 e_3 \) and \( b = b_1 e_1 + b_2 e_2 + b_3 e_3 \). Then compute \( e_1 \cdot (a \times b) = a_2 b_3 - a_3 b_2 \) directly from the formula for the cross product. Expanding the right-hand side of~\eqref{eq:V3-intrinsic-eq} using \( a \cdot e_1 = a_1 \), \( b \cdot e_1 = b_1 \), and \( a \cdot b = a_1 b_1 + a_2 b_2 + a_3 b_3 \) recovers the expression~\eqref{eq:V3-coord}.
\end{proof}

The four terms in~\eqref{eq:V3-intrinsic-eq} can be separated into a symmetric and a skew-symmetric part. Define
\begin{equation}
\label{eq:Spart}
S(a,b) = (a \cdot e_1)\,b + (b \cdot e_1)\,a - (a \cdot b)\,e_1 \quad\text{and}\quad [a, b] = 2\,(e_1 \cdot (a \times b))\,e_1,
\end{equation}
so that \( a \vm b = S(a,b) + \tfrac12[a,b] \), where \( [a,b] := a\vm b - b\vm a \) is the commutator. As shown in \cite{CEK}, \( [a, b] \) is a Lie bracket realizing the Heisenberg algebra \( \mathfrak{h}_1 \), with center \( \R e_1 \) and symplectic pair \( (e_2, e_3) \). On the other hand the symmetric part \( (\R^3, S) \) is the Vidinli-Jordan algebra \( \VJ_3 \), see Section 4 for the definition. Note that both of these sub-structures extend naturally to arbitrary dimensions, as we discuss in Sections~\ref{sec:defineV} and~\ref{sec:ViJo} whereas Vidinli algebras can only be extended to odd dimensions. 

The key to generalization is to recognize the last term of~\eqref{eq:V3-intrinsic-eq} as coming from a symplectic form. Set \( V_0 = e_1^\perp \) and define \( \omega_0 \colon V_0 \times V_0 \to \R \) by
\begin{equation}
\label{eq:omega0-def}
\omega_0(u, v) = e_1 \cdot (u \times v).
\end{equation}
Then \( \omega_0 \) is skew-symmetric and nondegenerate. Extending it to \( \R^3 \) by \( \omega_0(e_1, \cdot) = 0 \), Equation~\eqref{eq:V3-intrinsic-eq} becomes
\begin{equation}
\label{eq:V3-omega-form}
a \vm b = (a \cdot e_1)\,b + (b \cdot e_1)\,a - (a \cdot b)\,e_1 + \omega_0(a,b)\,e_1.
\end{equation}
The first three terms are determined by the inner product and the unit direction \( e_1 \) and encode the Jordan structure. The last is given by \( \omega_0 \) and encodes the Lie structure. The commutator \( [a, b] = 2\,\omega_0(a,b)\,e_1 \) and the form \( \omega_0 \) determine each other, and products inside \( V_0 \) reduce to \( a \vm b = (\omega_0(a,b) - a\cdot b)\,e_1 \), so the product of two vectors in \( e_1^\perp \) always lands on the principal axis \( \R e_1 \).

Equation~\eqref{eq:V3-omega-form} makes sense for any symplectic form on \( V_0 \) in place of \( \omega_0 \). Since any alternating form on \( V_0 \cong \R^2 \) is determined by the single value \( \omega(e_2, e_3) \), every such form is \( \omega = T\omega_0 \) for some \( T \in \R \) which is symplectic exactly when \( T \neq 0 \). This defines a one-parameter family. For each \( T \in \R \) define the \emph{twisted Vidinli multiplication}
\begin{equation}
\label{eq:twisted}
a \vmT{T} b = (a \cdot e_1)\,b + (b \cdot e_1)\,a - (a \cdot b)\,e_1 + T\,\omega_0(a,b)\,e_1,
\end{equation}
and write \( \mathcal{V}_T = (\R^3, \vmT{T}) \). The Vidinli algebra is \( \mathcal{V}_1 = \V_3 \), and \( \mathcal{V}_0 \) degenerates to the Vidinli-Jordan algebra \( \VJ_3 \) of Section~\ref{sec:ViJo}. By \cite[Theorem~2]{CEK}, \( \mathcal{V}_T \cong \mathcal{V}_{T'} \) as unital \( \R \)-algebras if and only if \( |T| = |T'| \). Therefore \( \{\mathcal{V}_T\}_{T > 0} \) are mutually non-isomorphic. We note this is a special case of Theorem 4.17.  None of the classical identities singles out \( T=1 \) within this family: associativity, the
alternative laws and flexibility fail for every \( T\neq0 \), and the Jordan identity holds only at
\( T=0 \). The normalization \( |T|=1 \) is metric in nature instead, as the following proposition
shows. Corollary~\ref{cor:normalization} below identifies the corresponding condition in all
dimensions.

\begin{proposition} \label{prop:metric-selects}
Among all twisted algebras \( \mathcal{V}_T \), the class \( |T|=1 \) is the unique one satisfying
\begin{equation}
\label{eq:metric-compat}
\|u \vmT{T} v\| = 1 \quad\text{for every orthonormal pair } u,v \in V_0 .
\end{equation}
\end{proposition}

\begin{proof}
For orthonormal \( u, v \in V_0 \) we have \( u \cdot e_1 = v \cdot e_1 = u \cdot v = 0 \), so \( u \vmT{T} v = T\,\omega_0(u,v)\,e_1 \). Since \( \{u,v\} \) is an orthonormal basis of \( V_0 \), \( |\omega_0(u,v)| = |\det[u\;v]| = 1 \), hence \( \|u \vmT{T} v\| = |T| \), and Condition~\eqref{eq:metric-compat} forces \( |T| = 1 \).
\end{proof}


\section{The Vidinli algebra in odd dimension}\label{sec:Vn}


The coordinate-free formula of Proposition \ref{prop:coordfree3} encodes both inner and cross product structures into a single unital bilinear operation. This formulation admits a natural extension to all odd dimensions. In this section, we define the general Vidinli algebra \( \mathbb{V}_{2n+1} \), and show how it canonically decomposes into a symmetric Jordan product and a skew-symmetric Lie bracket associated with a fixed direction. In this section we also determine the automorphism group and derivation algebra of $\Vn$.

\subsection{Definition and basic properties}\label{sec:defineV}

Let \( n\ge1 \) and let \( V = \mathbb{R}^{2n+1} \) be a Euclidean space equipped with the standard inner product. Let \(e_1, e_2,\ldots, e_{2n+1}\) be a fixed ordered orthonormal basis. Write \( V_0 := e_1^\perp \subset V \), so that \( V = \mathbb{R}e_1 \oplus V_0 \). Equip \( V_0 \) with the standard symplectic form \( \omega_0 \colon V_0 \times V_0 \to \mathbb{R} \), generalizing~\eqref{eq:omega0-def}, given by
$\omega_0(e_{2k},e_{2k+1}) = 1 = -\omega_0(e_{2k+1},e_{2k})\quad \text{for all } k=1,\ldots, n$
and \(\omega_0(e_k, e_l) = 0\) for all the other basis pairs.
We extend it to a bilinear form \( \omega \colon V \times V \to \mathbb{R} \) by setting
$\omega(e_1, \cdot) = 0 = \omega(\cdot, e_1)$.

We denote by \( J\colon V_0\to V_0 \) the linear map given by
\[
J(e_{2k}) = e_{2k+1},\qquad J(e_{2k+1}) = -e_{2k}\qquad (1\le k\le n).
\]
Note that \( J \) is orthogonal, that \( J^{\mathsf T}=-J \) and \( J^{2}=-I_{V_0} \), and that
\( \omega(u,v)=(Ju)\cdot v \) for all \( u,v\in V_0 \). With this notation, \( J \) makes
\( V_0 \) into a complex vector space, identified with \( \C^{n} \) by sending \( e_{2k} \) to the
\( k \)-th standard basis vector and \( e_{2k+1} \) to \( i \) times it.

\begin{definition}\label{defn:Vidinli}
The \emph{Vidinli multiplication} \( \vm := \vm_{e_1}^\omega \colon V \times V \to V \) is defined by
\[
a \vm b := (a \cdot e_1)\, b + (b \cdot e_1)\, a - (a \cdot b)\, e_1 + \omega(a, b)\, e_1.
\]
The resulting algebra \( \mathbb{V}_{2n+1} := (V, \vm) \) is called the \emph{Vidinli algebra of dimension \( 2n+1 \)}.
\end{definition}
The following basic properties follow from the definition. 
\begin{proposition}\label{pro:basics}
The Vidinli algebra \( \mathbb{V}_{2n+1} \) is unital with unit \( e_1 \), and is neither
commutative, associative, alternative nor flexible.
Every \( x\in\Vn \) satisfies the quadratic identity
\begin{equation}\label{eq:conic}
x \vm x \;=\; 2(x\cdot e_1)\,x-\|x\|^{2}e_1 ,
\end{equation}
and 
\begin{enumerate}
\item \( e_1 \) is the unique nonzero idempotent of \( \Vn \);
\item \( x\vm x\in\R_{\le0}\,e_1 \) if and only if \( x\in V_0 \), and then \( x\vm x=-\|x\|^{2}e_1 \);
\item every nonzero subalgebra of \( \Vn \) contains \( e_1 \);
\item \( \Vn \) is power-associative, the unital subalgebra generated by an element \( x \) being
\( \R e_1+\R x \).
\end{enumerate}
In particular \( \Vn \) is simple.
\end{proposition}

\begin{proof}
For any \( a \in \Vn \),
\[
e_1 \vm a = (e_1 \cdot e_1)\,a + (a \cdot e_1)\,e_1 - (e_1 \cdot a)\,e_1 + \omega(e_1, a)\,e_1 = a,
\]
and similarly $a\vm e_1 = a$. 
Hence \( e_1 \) is the unit. The non-commutativity follows since for any \(a,b\in\Vn\), we have
$a \vm b - b \vm a = 2\,\omega(a, b)\,e_1$,
which is nonzero whenever \( \omega(a,b)\neq0 \). Associativity and the alternative laws fail
together, since
\[
(e_2\vm e_2)\vm e_3=-e_3 \qquad\text{while}\qquad e_2\vm(e_2\vm e_3)=\omega(e_2,e_3)\,e_2 ,
\]
and flexibility fails since \( (e_2\vm e_3)\vm e_2=e_2 \) while \( e_2\vm(e_3\vm e_2)=-e_2 \).
Equation~\eqref{eq:conic} follows directly from the definition, since \( \omega(x,x)=0 \) so setting \( x=x_1e_1+x_0 \) with \( x_0\in V_0 \), we get \( x\vm x=(x_1^{2}-\|x_0\|^{2})e_1+2x_1x_0 \).

For (1), \( x\vm x=x \) implies \( 2x_1x_0=x_0 \) and \( x_1^{2}-\|x_0\|^{2}=x_1 \). If \( x_0\neq0 \), the first equation gives \( x_1=\tfrac12 \) and the second \( \|x_0\|^{2}=-\tfrac14<0 \), which is impossible. Hence \( x_0=0 \), and then \( x_1^{2}=x_1 \) implies \( x=e_1 \). For (2), suppose first that \( x\in V_0 \). Then \( x\cdot e_1=0 \), so that \eqref{eq:conic} gives \( x\vm x=-\|x\|^{2}e_1\in\R_{\le0}\,e_1 \). Conversely, suppose that \( x\vm x\in\R_{\le0}\,e_1 \). In particular \( x\vm x\in\R e_1 \), so that
\( x_1x_0=0 \). If \( x_0=0 \), then \( x\vm x=x_1^{2}e_1 \), whose coefficient is \( \le0 \) only for \( x_1=0 \), and then \( x=0\in V_0 \). If \( x_1=0 \), then \( x\in V_0 \) and \( x\vm x=-\|x_0\|^{2}e_1 \). For (3), let \( B \) be a nonzero subalgebra and \( 0\neq x\in B \). Then by Equation \eqref{eq:conic}, \( x\vm x-2(x\cdot e_1)\,x=-\|x\|^{2}e_1\in B \) with \( \|x\|^{2}\neq0 \). Therefore \( e_1\in B \).

For (4), note that \( \R e_1+\R x \) is closed under \( \vm \) by Equation~\eqref{eq:conic}. For
\( x\in\R e_1 \) this subalgebra is \( \R e_1\cong\R \), and for \( x\notin\R e_1 \) the assignment
\( X\mapsto x \) induces an isomorphism
\[
\R[X]/\bigl(X^{2}-2(x\cdot e_1)X+\|x\|^{2}\bigr)\;\longrightarrow\;\R e_1+\R x
\]
of unital algebras. In either case \( \R e_1+\R x \) is commutative and associative, and hence so is
every subalgebra generated by a single element.

Finally, a nonzero two-sided ideal \( I \) is in particular a nonzero subalgebra, so \( e_1\in I \) by (3) and \( I=\mathbb{V}_{2n+1} \). Hence \( \Vn \) is simple, as required.
\end{proof}

Recall that the \emph{centroid} of an algebra \( A \) is the set of all endomorphisms \( \theta\in\End_\R(A) \) satisfying \( \theta(a\vm b)=\theta(a)\vm b=a\vm\theta(b) \) for all \( a,b \). Recall also that the \emph{multiplication algebra} \( \mathrm{M}(A)\subseteq\End_\R(A) \) is the unital subalgebra
generated by all left and right multiplication operators \( L_a,R_a \). We refer to \cite{Schafer} for these notions and for the general theory of non-associative algebras.

\begin{proposition}\label{prop:central}
The centroid of \( \Vn \) is the one-dimensional subspace \( \R\,\cdot\mathrm{id} \). In particular \( \Vn \) is a central simple real algebra, and
its multiplication algebra is 
\[
\mathrm{M}(\Vn)\;=\;\End_\R(\Vn)\;\cong\;M_{2n+1}(\R).
\]
\end{proposition}

\begin{proof}
Let \( \theta \) be in the centroid and set \( z:=\theta(e_1) \). Applying the defining property to
\( e_1\vm a=a=a\vm e_1 \), we obtain \( \theta(a)=z\vm a \) and \( \theta(a)=a\vm z \) for every \( a \).
Thus \( z\vm a-a\vm z=2\,\omega(z,a)\,e_1=0 \) for all \( a \). Since \( \omega \) is nondegenerate on
\( V_0 \) and \( \omega(e_1,\cdot)=0 \), this forces \( z=\lambda e_1 \) for some \( \lambda\in\R \), and
then \( \theta(a)=\lambda a \). Together with Proposition~\ref{pro:basics} this says that \( \Vn \) is
central simple.

Over an arbitrary field the centroid and the multiplication algebra of a Vidinli algebra are
determined in \cite{ElduqueRandez}; we give the real case.
For the multiplication algebra, note that the \( \mathrm M(\Vn) \)-submodules of \( \Vn \) are
exactly its ideals, so that \( \Vn \) is an irreducible \( \mathrm M(\Vn) \)-module, and
\( \End_{\mathrm M(\Vn)}(\Vn) \) is the centroid, which is \( \R \) by the first part. Hence
\( \mathrm M(\Vn)=\End_\R(\Vn) \) by the Jacobson density theorem, as required. We refer to
\cite{Schafer} for the multiplication algebra and the centroid in the general non-associative setting.
\end{proof}

Next we determine the structure constants of the Vidinli algebra $\Vn$ and a coordinate formula both of which are practical for computations. Let $e_i, e_j, e_k$ be basis elements. We write $\mathrm{v}_{i,j}^k$ for the coefficient of $e_k$ in the product \(e_i\vm e_j\) so that
\[
e_i\vm e_j = \sum_{k=1}^{2n+1} \mathrm{v}_{i,j}^k e_k.
\]
Since the product \(e_i\vm e_j\) is contained in the span of $e_1, e_i, e_j$, we have that
$\mathrm{v}_{i,j}^k = 0\text{ if } k\not\in\{1,i,j\}$.
The remaining structure constants can be grouped in three classes and evaluated using Definition \ref{defn:Vidinli}.

\smallskip

\noindent (1) If \(i=j\neq1\), then 
    $e_i\vm e_i = - e_1.$ Hence 
   $ \mathrm{v}_{i,i}^1 = -1, \quad
    \mathrm{v}_{i,i}^i = 0$,
   
   \smallskip
   
\noindent (2) If $i =1$ or $j=1$, then since $e_1$ is the unit, 
   $\mathrm{v}_{1,j}^k = \delta_{j,k},\quad 
   \mathrm{v}_{i,1}^k = \delta_{i,k}$,
   
   \smallskip
   
\noindent (3) If $i\neq j$ and $1\not\in\{i,j\}$, then $e_i\cdot e_1 = e_j\cdot e_1 = e_i\cdot e_j = 0$ and hence
  \[
  \mathrm{v}_{i,j}^k = 
  \begin{cases}
    \omega(e_i, e_j), & \text{if } k=1 \\
    0, & \text{otherwise.}
  \end{cases}
\]

\begin{corollary}\label{cor:coordinates}
The Vidinli multiplication given in Definition~\ref{defn:Vidinli} can be written as $a \vm b := a_1 b + b_1 a + (\omega(a, b)-(a \cdot b)) e_1$ where \( a = \sum_{i=1}^{2n+1} a_i e_i \), \( b = \sum_{i=1}^{2n+1} b_i e_i \).  Substituting the symplectic form and the dot product, we obtain the coordinate formula
\[
a \vm b = c_1 e_1 + \sum_{i=2}^{2n+1} c_i e_i,
\]
where
\[
c_1 = a_1 b_1 - \sum_{i=2}^{2n+1} a_i b_i + \sum_{k=1}^n (a_{2k} b_{2k+1} - a_{2k+1} b_{2k}),
\quad 
c_i = a_1 b_i + b_1 a_i \quad \text{for } i = 2, \dots, 2n+1.
\]
\end{corollary}

For \( n=1 \), the resulting structure constants are those of Definition~\ref{def:V3}. Hence we recover the three-dimensional Vidinli algebra \( \V_3 \).

\begin{example}[The Vidinli algebra \( \mathbb{V}_5 \)]
The structure constants for the 5-dimensional Vidinli algebra are
\[
e_i \vm e_i = -e_1 \ (i = 2, \dots, 5);\;
e_2 \vm e_3 = e_4 \vm e_5 = e_1;\;
e_3 \vm e_2 = e_5 \vm e_4 = -e_1 ,
\]
all remaining products of distinct basis vectors of \( V_0 \) being zero.
\end{example}

\begin{remark}\label{rem:crossproducts}
The algebras \( \V_3 \) and \( \V_7 \) are distinguished by their connection with the cross products.
Both admit cross product structures on their underlying real vector spaces.  Also in both of them the symplectic
form defining the algebra is induced by the corresponding cross product. By the theorem of Brown and
Gray \cite{Brown}, a nontrivial cross product exists only in dimensions \( 3 \) and \( 7 \).
\end{remark}

\subsection{Automorphisms and derivations}

We now determine the full automorphism group of the Vidinli algebra \(\mathbb{V}_{2n+1}\). The key observation is Proposition~\ref{pro:basics}(1) which shows that the unique nonzero idempotent is \(e_1\). Thus every automorphism fixes the unit direction. The restriction to the symplectic hyperplane \(V_0 = e_1^\perp\) then preserves both the inner product and the skew form \(\omega\). Hence it lies in the intersection
\[
\Or(2n)\;\cap\; \Sp(2n,\R) \;\cong\; \U(n).
\]

\begin{theorem}
\label{thm:autV}
The automorphism group of the Vidinli algebra \(\mathbb{V}_{2n+1} = (\mathbb{R}e_1 \oplus V_0, \vm)\) is isomorphic to the unitary group
$\Aut(\mathbb{V}_{2n+1}) \;\cong\; \U(n)$.
\end{theorem}
\begin{proof}
Let \( \Phi \) be an automorphism of \( \V_{2n+1} \). By Proposition~\ref{pro:basics}(1) the element \( e_1 \) is the unique nonzero idempotent, so \( \Phi(e_1)=e_1 \). For \( v\in V_0 \) we have \( v\vm v=-\|v\|^{2}e_1 \) by Proposition~\ref{pro:basics}(2), hence
\[
\Phi(v)\vm\Phi(v)=\Phi(v\vm v)=-\|v\|^{2}e_1\in\R_{\le0}\,e_1 ,
\]
and by the same part of that proposition, \( \Phi(v)\in V_0 \) together with \( \|\Phi(v)\|=\|v\| \). Thus \( \Phi \) restricts to a linear isometry
\( \psi = \Phi|_{V_{0}} \colon V_{0} \to V_{0} \), that is, \( \psi\in \Or(V_0,\cdot) \).

Next we show that \( \psi \) preserves \( \omega_0 \). Note that
\( a\vm b=(\omega_0(a,b)-a\cdot b)\,e_1 \) for \( a,b\in V_0 \), and that \( \Phi(e_1)=e_1 \).
Applying \( \Phi \) to this product, we obtain
\[
\bigl(\omega_0(a,b)-a\cdot b\bigr)e_1=\Phi(a\vm b)=\Phi(a)\vm\Phi(b)
=\bigl(\omega_0(\psi a,\psi b)-\psi a\cdot\psi b\bigr)e_1 .
\]
Here the last equality holds since \( \psi(a) \) and \( \psi(b) \) lie in \( V_0 \). As \( \psi \) is
an isometry, the two inner product terms agree, whence
\( \omega_0(\psi a,\psi b)=\omega_0(a,b) \), that is, \( \psi \in \Sp(V_0, \omega_0) \). We conclude that
$\psi \in \Or(V_0, \cdot) \cap \Sp(V_0, \omega_0)$.
This intersection is isomorphic to the unitary group \( \U(n) \), see~\cite[Proposition~1.6.5]{Koszul}.

Conversely, every \( \psi \in \U(n) = \Or(V_0,\cdot)\cap\Sp(V_0,\omega_0) \) extends to an automorphism. Define \( \Phi\colon\Vn\to\Vn \) by \( \Phi(a_1 e_1 + a_0) = a_1 e_1 + \psi(a_0) \) for \( a_0\in V_0 \). Writing \( a = a_1 e_1 + a_0 \), \( b = b_1 e_1 + b_0 \) with \( a_0,b_0\in V_0 \), we get
$a \vm b = (a_1 b_1 - a_0\cdot b_0 + \omega_0(a_0,b_0))\,e_1 + a_1 b_0 + b_1 a_0$,
and since \( \psi \) preserves both the inner product and \( \omega_0 \), it preserves each term, so \( \Phi(a\vm b) = \Phi(a)\vm\Phi(b) \). Hence \( \Phi \) is an automorphism, and \( \Aut(\Vn)\cong \U(n) \), as required.\end{proof}

The same argument at the infinitesimal level determines the derivation algebra. Over an arbitrary
field this is also obtained in \cite{ElduqueRandez}.

\begin{proposition}\label{pro:der}
The derivation algebra of \( \Vn \) consists of the linear maps vanishing on \( e_1 \) and
restricting on \( V_0 \) to an element of \( \so(V_0)\cap\spp(V_0,\omega) \). In particular
$\Der(\Vn)\cong\uu(n)$.
\end{proposition}

\begin{proof}
Let \( D \) be a derivation of \( \Vn \). Applying \( D \) to \( e_1\vm e_1=e_1 \), we get
\( D(e_1)=2D(e_1) \) and hence \( D(e_1)=0 \). Let \( u\in V_0 \) and write
\( D(u)=\lambda(u)\,e_1+d(u) \) with \( d(u)\in V_0 \). Applying \( D \) to
\( u\vm u=-\|u\|^{2}e_1 \) and using \( D(e_1)=0 \), we obtain
\[
0=D(u)\vm u+u\vm D(u)=2\bigl(\lambda(u)\,u-(d(u)\cdot u)\,e_1\bigr).
\]
Here the second equality holds since \( u\cdot e_1=0 \), so that only the symmetric part of the
product contributes. Therefore \( \lambda=0 \) and \( d(u)\cdot u=0 \) for every
\( u \), that is, \( D(V_0)\subseteq V_0 \) and \( d:=D|_{V_0} \) is skew-symmetric. Observe next that
\( u\vm v=(\omega(u,v)-u\cdot v)\,e_1 \) for \( u,v\in V_0 \), so that \( D(u\vm v)=0 \) and
\begin{equation}\label{eq:der}
0=D(u)\vm v+u\vm D(v)=\bigl(\omega(du,v)+\omega(u,dv)-(du\cdot v+u\cdot dv)\bigr)e_1 .
\end{equation}
Here the last bracket vanishes since \( d \) is skew-symmetric. Hence
\( \omega(du,v)+\omega(u,dv)=0 \) and \( d\in\so(V_0)\cap\spp(V_0,\omega) \).

Conversely, both sides of the derivation identity are bilinear, so it suffices to verify it on pairs
of elements of \( \{e_1\}\cup V_0 \). The pairs involving \( e_1 \) are immediate from
\( D(e_1)=0 \) and unitality, and the pairs in \( V_0 \) are Equation~\eqref{eq:der}. Thus every
\( d\in\so(V_0)\cap\spp(V_0,\omega) \), extended by \( D(e_1)=0 \), is a derivation. Finally, the
intersection of two closed subgroups has the intersection of their Lie algebras as its Lie algebra.
Hence \( \so(V_0)\cap\spp(V_0,\omega) \) is the Lie algebra of the unitary group
\( \Or(V_0,\cdot)\cap\Sp(V_0,\omega)\cong\U(n) \) of Theorem~\ref{thm:autV}, that is, \( \uu(n) \),
as required.
\end{proof}

The Vidinli algebras carry a nontrivial abelian grading in every dimension, coming from the
diagonal torus of the automorphism group.

\begin{proposition}\label{prop:torus-grading}
For \( 1\le j\le n \) put \( L_j:=\Span_\R\{e_{2j},e_{2j+1}\}\subset V_0 \), and let \( \chi_1,\dots,\chi_n \) be the standard generators of \( (\mathbb{Z}/2)^n \). Then
\[
\V_{2n+1}=\R e_1\oplus L_1\oplus\cdots\oplus L_n,
\qquad \deg(\R e_1)=0,\quad \deg(L_j)=\chi_j ,
\]
is a grading of \( \V_{2n+1} \) by \( (\mathbb{Z}/2)^n \), and it is fine, that is, it admits no proper refinement by a group grading. It is the eigenspace decomposition of the diagonal subgroup \( \{\pm1\}^n\subset \U(n)=\Aut(\V_{2n+1}) \).
\end{proposition}

\begin{proof}
For \( u,v\in V_0 \) one has \( u\vm v=(\omega(u,v)-u\cdot v)\,e_1 \). If \( i\neq j \), then \( L_i\perp L_j \) and \( \omega(L_i,L_j)=0 \), so \( L_i\vm L_j=0 \), which lies in every component; if \( i=j \), then \( L_i\vm L_i\subseteq\R e_1 \), the component of degree \( \chi_i+\chi_i=0 \). Finally \( e_1 \) is the unit, so \( \R e_1\vm L_j=L_j \), of degree \( 0+\chi_j \). Hence the decomposition is a grading.

For fineness, note that a proper refinement must split some \( L_j \), since \( \R e_1 \) is
one-dimensional. Suppose then that \( L_j \) splits as \( (L_j\cap B_g)\oplus(L_j\cap B_h) \) with \( g\neq h \) in a refining grading \( \bigoplus_{k}B_k \) by a group written multiplicatively, not necessarily abelian. Pick nonzero \( x\in L_j\cap B_g \) and \( y\in L_j\cap B_h \); they are linearly independent, hence a basis of \( L_j \), so \( \omega(x,y)\neq0 \) because \( \omega|_{L_j} \) is nondegenerate. Now \( x\vm y=(\omega(x,y)-x\cdot y)e_1 \) and \( y\vm x=(-\omega(x,y)-x\cdot y)e_1 \) cannot both vanish, since that would imply \( \omega(x,y)=0 \). A nonzero product lands in \( \R e_1 \), which is homogeneous of degree \( 1 \), the identity element
of the grading group. Indeed, we write \( e_1=\sum_k u_k \) with \( u_k\in B_k \). For a homogeneous
\( z\in B_l \), comparing the components of degree \( l \) in \( z=e_1\vm z=\sum_k u_k\vm z \) gives
\( z=u_1\vm z \), and similarly \( z=z\vm u_1 \). Hence \( u_1 \) is a unit of the algebra, so that
\( u_1=e_1 \). It follows that \( gh=1 \) or \( hg=1 \), and in either case \( h=g^{-1} \). Squaring,
we get \( x\vm x=-\|x\|^2e_1\neq0 \), so \( g^{2}=1 \) hence we get \( h=g^{-1}=g \), a contradiction.

The last assertion is immediate, since the element of \( \{\pm1\}^n \) with \( -1 \) in the \( j \)-th slot acts as \( -1 \) on \( L_j \) and as \( +1 \) on the remaining summands.
\end{proof}

\subsection{Jordan and Lie structures}\label{sec:ViJo}

The Vidinli multiplication splits into symmetric and skew-symmetric components, as
in~\eqref{eq:Spart}. We write
$a \vm b\ =\ a \vijo b\ +\ \tfrac12\,[a,b]$.
Here $a \vijo b\ :=\ (a \cdot e_1)\, b \ +\ (b \cdot e_1)\, a\ -\ (a \cdot b)\, e_1$
is the symmetric part of the Vidinli multiplication, one half of its anti-commutator, and
$[a,b] := a\vm b - b\vm a = 2\,\omega(a, b)\, e_1$
is the commutator. With this notation, the skew-symmetric part of the product is
\( \tfrac12[a,b]=\omega(a,b)e_1 \). We use \( [\,\cdot,\cdot\,] \) for the commutator throughout. This decomposition exhibits \( \Vn \) as a non-associative algebra that is simultaneously \emph{Jordan-admissible} and \emph{Lie-admissible}. The symmetrization \( \vijo \) is a Jordan algebra (Proposition~\ref{pro:VJ}) and the commutator is a Lie bracket (Proposition~\ref{thm:Heisenberg}). As mentioned in \cite{McCrimmon2}, structures of this kind are studied by physicists to find systems containing both bosons and fermions, see \cite{Okubo} for Lie-admissible algebras in this context. We consider the two induced structures separately.

Let $V = \R e_1 \oplus V_0$ be as above. We call the space $V$ equipped with the product 
$\vijo$ the \emph{Vidinli-Jordan algebra}, denoted $\VJ_{2n+1}$. Note that the anti-commutator can be defined in any dimension. We extend the definition of Vidinli-Jordan algebras to arbitrary dimensions as follows.

Let \( m\ge1 \) and let \( V = \mathbb{R}^m \) be a real inner product space with the standard orthonormal basis \( \{e_1, e_2, \dots, e_m\} \). Fix the unit vector \( e_1 \in V \) and decompose
$V = \mathbb{R}e_1 \oplus V_0$ where $V_0 := e_1^\perp$.
Define a symmetric bilinear product \( \vijo \colon V \times V \to V \) by
\[
a \vijo b := (a \cdot e_1)\,b + (b \cdot e_1)\,a - (a \cdot b)\,e_1.
\]
The algebra \( \VJ_m := (V, \vijo) \) is called the \emph{Vidinli–Jordan algebra} of dimension \( m \) with principal axis \( e_1 \). The following proposition collects basic properties. 

\begin{proposition}\label{pro:VJ}
For any \( m \geq 1 \), the product \( \vijo \) is bilinear, commutative and unital, with unit \( e_1 \), and it satisfies the Jordan identity
$(a \vijo a) \vijo (b \vijo a) = ((a \vijo a) \vijo b) \vijo a$ for all \( a, b \in V \). Hence \( \VJ_m \) is a Jordan algebra and is in particular power-associative. It is simple for every \( m\ge1 \), and not formally real for every \( m\ge2 \). It is non-associative
exactly for \( m\ge3 \); the two smallest cases, \( \VJ_1=\R \) and \( \VJ_2\cong\C \), are associative.
\end{proposition}

\begin{proof}
The listed properties can be confirmed by direct computation. For the Jordan identity, we follow the computation in \cite[Ch.~III, §6]{Koecher}, using the expression
$a\vijo b = \lambda(a)\,b + \lambda(b)\,a - \mu(a,b)\,e_1$,
with \( \lambda(a) = a \cdot e_1 \), \( \mu(a,b) = a \cdot b \).
For \( m\ge2 \) the resulting Jordan algebra is not formally real, since \( e_1 \vijo e_1 + e_2 \vijo e_2 = 0 \) while \( e_2 \vijo e_2 = -e_1 \ne 0 \). Note that this argument requires \( m\ge2 \), and indeed \( \VJ_1=\R \) is formally real. On the other hand, associativity holds for one step further. \( \VJ_2=\R e_1\oplus\R e_2 \) with \( e_2\vijo e_2=-e_1 \) is a copy of \( \C \), whereas for \( m\ge3 \) we have \( (e_2\vijo e_2)\vijo e_3=-e_3\neq 0=e_2\vijo(e_2\vijo e_3) \).
Finally, simplicity follows by the argument of Proposition~\ref{pro:basics}. Indeed, the identity~\eqref{eq:conic} involves only the symmetric part of the product, so it holds for \( \vijo \) as well, and every nonzero subalgebra of \( \VJ_m \) therefore contains \( e_1 \).
\end{proof}

The Vidinli-Jordan algebra \( \VJ_m \) is the Jordan algebra of a symmetric bilinear form, that is, a
Jordan algebra of Clifford type. The form in question is \emph{negative} definite on \( V_0 \), since
\( u\vijo u=-\|u\|^2e_1 \); in the standard classification \( \VJ_m \) is therefore the Jordan
algebra \( J(\R^{0,m-1}) \) of the negative definite symmetric bilinear form of rank \( m-1 \), the
anti-Euclidean spin factor. Accordingly, for \( m\ge2 \), \( \VJ_m \) is not the Euclidean spin
factor \( J\mathrm{Spin}_{m-1} \), which is formally real. We use ``spin factor'' throughout as a
synonym for ``Jordan algebra of a symmetric bilinear form'', with no definiteness assumption, and
reserve ``Euclidean spin factor'' for the formally real case. 

Next we consider the commutator algebra of the Vidinli algebra.

\begin{proposition}\label{thm:Heisenberg}
Let $\V_{2n+1}$ be the Vidinli algebra of dimension $2n+1$, with unit vector $e_1$ and hyperplane $V_0 = e_1^\perp$. Then:

\smallskip
\noindent (1) For $a,b \in V_0$, one has $[a,b] = 2\,\omega(a,b)\,e_1$, and $[e_1,a] = 0$ for all $a \in \V_{2n+1}$.
  
\smallskip
\noindent (2) The bracket $[\cdot,\cdot]$ satisfies the Jacobi identity and defines a $2$-step nilpotent Lie algebra on $\V_{2n+1}$ with center $\R e_1$.

\smallskip
\noindent (3) This Lie algebra \((\Vn, [\cdot, \cdot])\) is isomorphic to the real Heisenberg Lie algebra $\mathfrak{h}_n$ of dimension $2n+1$, with center $\R e_1$ and subspace $V_0$ satisfying $[u,v] = 2\,\omega(u,v)\,e_1$.
\end{proposition}

\begin{proof}
Part (1) is straightforward. As the image of $[\cdot,\cdot]$ lies in the one-dimensional subspace $\R e_1$, which is central, the bracket is automatically $2$-step nilpotent and satisfies the Jacobi identity. Choosing a Darboux basis $\{p_i,q_i\}_{i=1}^n$ for $(V_0,\omega)$, we get the relations
$[p_i,q_j] = 2\,\delta_{ij}\,e_1$, $[p_i,p_j]=[q_i,q_j]=0$, and $[V_0,e_1]=0$,
which are exactly the defining relations of $\mathfrak{h}_n$ after rescaling $z=2e_1$.
\end{proof}
\medskip
The commutator bracket $[a,b]$ is defined in all dimensions $2n+1$.  
However, it has a particularly direct geometric description in dimensions $3$ and $7$. Note that these are the only Euclidean dimensions in which the standard symplectic form on the hyperplane $V_0$ can be realized as the $e_1$–projection of a vector cross product \cite{Brown}, \cite{Salamon}.  
In these cases, the commutator is simply twice the $e_1$–component of the corresponding cross product, making the connection between the Vidinli multiplication and the geometry of $\R^3$ and $\R^7$ especially transparent.

\begin{proposition}\label{prop:rigidity}
Let \( \ast \) be a bilinear product on \( \R^{2n+1} \) such that \( e_1 \) is a unit,
\( u\ast u=-e_1 \) for every unit vector \( u\in V_0 \), and, for distinct \( j,k\ge2 \),
\( e_j\ast e_k\in\R e_1 \), with \( e_j\ast e_k=0 \) unless \( \{j,k\}=\{2i,2i+1\} \) for some
\( 1\le i\le n \). Then \( \ast \) is determined by the \( n \) parameters \( \delta_1,\dots,\delta_n \) defined by
\[
e_{2i}\ast e_{2i+1}=\delta_i e_1,\qquad\text{the reversed products being }
e_{2i+1}\ast e_{2i}=-\delta_i e_1 .
\]
Moreover, if \( \delta_i\in\{\pm1\} \) for every \( i \), then \( (\R^{2n+1},\ast)\cong\Vn \). In
particular all \( 2^n \) sign choices give the same algebra up to isomorphism, \( \Vn \) being the
case \( \delta_1=\cdots=\delta_n=1 \).
\end{proposition}

\begin{proof}
The square hypothesis forces skew-symmetry off the diagonal. Indeed, for distinct \( j,k\ge2 \) the vector
\( (e_j+e_k)/\sqrt2 \) is a unit vector of \( V_0 \), so expanding
\( \tfrac12(e_j+e_k)\ast(e_j+e_k)=-e_1 \) and using \( e_j\ast e_j=e_k\ast e_k=-e_1 \) implies
\( e_j\ast e_k+e_k\ast e_j=0 \). The products involving \( e_1 \) are fixed by unitality, the squares
by hypothesis, and the remaining unordered pairs \( \{j,k\} \) with \( j,k\ge2 \) vanish except for the
\( n \) exceptional ones, whose values are the \( \delta_i \); so \( \ast \) is determined.

If every \( \delta_i\in\{\pm1\} \), then \( \ast \) is the Vidinli-type product attached to the
alternating form \( \tilde\omega \) with \( \tilde\omega(e_{2i},e_{2i+1})=\delta_i \) and all other
pairings of basis vectors zero. The linear map fixing \( e_1 \) and transposing
\( e_{2i}\leftrightarrow e_{2i+1} \) for each \( i \) with \( \delta_i=-1 \) is orthogonal and carries
\( \tilde\omega \) to \( \omega_0 \), hence is an isomorphism onto \( \Vn \).
\end{proof}

\begin{example}[The $3$-dimensional case]
Let $\{e_1,e_2,e_3\}$ be an orthonormal basis for $\R^3$ with $V_0 = \Span_\R\{e_2,e_3\}$ and $\omega_0 = e^{23}$, see \cite{Salamon} for the notation. Define the usual vector cross product by
\[
u \times v = 
(u_2 v_3 - u_3 v_2, \
- u_1 v_3 + u_3 v_1, \
u_1 v_2 - u_2 v_1)
\]
The associated $3$-form is $\varphi = e^{123}$, and
$\iota_{e_1}\varphi = e^{23} = \omega_0$.
Thus for $v,w \in V_0$, the commutator is
$[v,w] = 2\,\omega_0(v,w)\,e_1 = 2\,(e_1 \cdot (v \times w))\,e_1$,
recovering the $e_1$–projection of the $3$-dimensional cross product.
\end{example}

\begin{example}[The $7$-dimensional case]\label{ex:standardCross}
Let $\{e_1,\dots,e_7\}$ be an orthonormal basis for $\R^7$ with $V_0 = e_1^\perp$ and $\omega_{0}$ the standard symplectic form on $V_0$. Let $\times$ be the cross product on $\R^7$ induced by octonion multiplication, with associated calibration $3$-form
\[
\varphi = e^{123} + e^{145} + e^{167} + e^{246} - e^{257} - e^{347} - e^{356};
\]
see \cite{Salamon} for explicit formulas. Then $\iota_{e_1}\varphi = e^{23} + e^{45} + e^{67} = \omega_0$, all other pairings $\omega_0(e_i,e_j)$ of distinct basis vectors of $V_0$ being zero. Once more the commutator agrees with twice the $e_1$–projection of the cross product.\end{example}

\subsection{Dependence on the symplectic form}

In Definition \ref{defn:Vidinli}, we have fixed an inner product and the standard symplectic form. However the algebra structure depends on the choice of the symplectic form \( \omega \) on the hyperplane \( V_0 = e_1^\perp \cong \mathbb{R}^{2n} \).
We make the following definition for future reference.
\begin{definition}\label{defn:vidtype}
    Let $m\ge1$, let $V=\R^m=\R e_1\oplus V_0$ be as in Section~\ref{sec:ViJo}, and let $\tilde\omega$ be the trivial extension of an arbitrary alternating bilinear form $\tilde\omega_0$ on $V_0$, not necessarily nondegenerate. Define 
    \[
    a\vm^{\tilde\omega}_{e_1} b = (a \cdot e_1)\, b + (b \cdot e_1)\, a - (a \cdot b)\, e_1 + \tilde\omega(a, b)\, e_1.
    \]
    We call the resulting algebra $\V^{\tilde\omega}_m = (V, \vm^{\tilde\omega}_{e_1})$ a \emph{Vidinli-type algebra} associated to $\tilde\omega$, and we write $\Vn^{\tilde\omega}$ in the case $m=2n+1$. We say it is of \emph{symplectic type} when $\tilde\omega_0$ is nondegenerate. Note that the degenerate case also appears among subalgebras of $\Vn$.
\end{definition}

\begin{remark}\label{rem:type-squares}
The identity~\eqref{eq:conic} and its consequences (1)-(4) involve only the symmetric part of the
product, which does not depend on the form. Hence \eqref{eq:conic} and (1)-(4) hold in every
Vidinli-type algebra \( \Vn^{\tilde\omega} \), for every alternating \( \tilde\omega \), and so does
simplicity, since its proof used only (3). Note that the computation in Proposition~\ref{pro:basics} gives
\( e_2\vm^{\tilde\omega}_{e_1}(e_2\vm^{\tilde\omega}_{e_1}e_3)=\tilde\omega(e_2,e_3)\,e_2 \) but
\( (e_2\vm^{\tilde\omega}_{e_1}e_2)\vm^{\tilde\omega}_{e_1}e_3=-e_3 \), and that \( -e_3 \) is never a
scalar multiple of \( e_2 \). Hence non-associativity and the failure of the alternative laws persist
for every \( \tilde\omega \). On the other hand \( \Vn^{\tilde\omega} \) is commutative, and
then also flexible, if and only if \( \tilde\omega=0 \). For \( \tilde\omega\neq0 \), any pair
\( u,v \) with \( \tilde\omega(u,v)\neq0 \) violates flexibility. In particular every Vidinli-type algebra is simple, including
\( \VJ_{2n+1}=\Vn^{0} \). Moreover \( e_1 \) is the unique nonzero idempotent of \( \Vn^{\tilde\omega} \) and \( V_0=\{x : x\vm^{\tilde\omega}_{e_1}x\in\R_{\le0}\,e_1\} \), so both the
distinguished axis and the Euclidean structure of \( V_0 \) are determined by the algebra alone.
\end{remark}

It is a classical fact that any two symplectic forms on \( V_0 \) are symplectomorphic, that is, for
any such \( \omega_1 \) and \( \omega_2 \) there exists an invertible linear map
\( T\colon V_0\to V_0 \) such that \( \omega_2(Tu,Tv)=\omega_1(u,v) \) for all \( u,v\in V_0 \).
However, such a symplectomorphism need not preserve the inner product on \( V_0 \), and thus does not
in general induce an isomorphism of the corresponding Vidinli-type algebras. These algebras are classified in Appendix~A (Theorem~A.1). The classification follows from the description of Vidinli algebras by bilinear forms in~\cite{ElduqueRandez}, under which an isomorphism of Vidinli algebras is an isometry of the associated bilinear forms, combined with the orthogonal normal form of a real skew-symmetric matrix. Our main results do not depend on it, but it makes it easy to produce non-isomorphic Vidinli-type algebras of symplectic type; see Example~A.3.

To work canonically, throughout this paper, we fix once and for all the standard symplectic form
$\omega_0(e_{2i}, e_{2i+1}) = 1$ for $i = 1, \dots, n$,
with respect to an ordered orthonormal basis \( \{e_1, \dots, e_{2n+1}\} \) of \( \mathbb{R}^{2n+1} \). Whenever we consider an inclusion $\R^{2n+1}\subset \R^{2n+3}$, we assume that the basis of $\R^{2n+3}$ extends that of $\R^{2n+1}$ by appending the ordered pair $(e_{2n+2}, e_{2n+3})$ to the end of the fixed basis.
\begin{remark}\label{rem:chain}
With this convention \( \Vn \) is the subalgebra \( \R e_1\oplus\Span_\R\{e_2,\dots,e_{2n+1}\} \) of
\( \V_{2n+3} \), by Theorem~\ref{thm:sub-lattice} below. We thus obtain a chain
\( \V_3\subset\V_5\subset\V_7\subset\cdots \) of unital subalgebras. Their union is a Vidinli-type
algebra of countable dimension, and it is again simple, since every nonzero element \( x \) of it
satisfies \( x\vm x-2(x\cdot e_1)x=-\|x\|^2e_1 \) with \( \|x\|^2\neq0 \), so that any nonzero ideal
contains \( e_1 \).
\end{remark}
\section{Left multiplication, zero divisors and inverses}\label{sec:leftmult}

We now turn to the left multiplication operators \( L_a \) for $a\in \Vn$. The characteristic polynomial of \( L_a \) factors explicitly, showing that the space of zero divisors is exactly the symplectic hyperplane \( V_0 = e_1^\perp \), and that bijectivity of \( L_a \) is governed by the projection onto the unit axis. We also introduce a conjugation map and use it to determine two-sided inverses of nonzero vectors; these turn out to exist for every nonzero element, but to be unique only outside \( V_0 \) when \(n > 1\).

We first evaluate the characteristic polynomial of the linear map \( L_a\) for any \( a \in \V_{2n+1} \). Similar results hold for the right multiplication maps $R_b$.

\begin{proposition}\label{prop:charpoly}
Let \( a = a_1 e_1 + u \in \mathbb{V}_{2n+1} \), where \( u \in V_0 := e_1^\perp \). Then 

\smallskip
\noindent (1) The characteristic polynomial of \( L_a \) is
$\chi_{L_a}(\lambda) = (\lambda - a_1)^{2n-1} \left( (\lambda-a_1)^2 + \|u\|^2 \right)$.

\smallskip
\noindent (2) The determinant of \( L_a\) is 
$\det(L_a) = a_1^{\, 2n-1} \, \|a\|^2$.

\smallskip
\noindent (3) The same formulas hold for the right multiplication \( R_a\); in particular \( \det(R_a)=\det(L_a) \), and \( L_a \) is bijective if and only if \( R_a \) is.
\end{proposition}

\begin{proof}
The left multiplication map \( L_a \) acts on \( \V_{2n+1} = \mathbb{R} \oplus V_0 \) by
\[
L_a({e}_1) = a, \qquad
L_a({v}) = a_1 {v} + (\omega_0({u}, {v}) - {u} \cdot {v}) {e}_1
\quad \text{for } {v} \in V_0.
\]
Its matrix with respect to this decomposition takes the block form 
\[
L_a =
\begin{pmatrix}
a_1 & {r} \\
{u} & a_1 I_{2n}
\end{pmatrix}
\]
where \( r = ({r}_j)_j = (\omega_0({u}, {e}_j) - u\cdot e_j)_j \) is a row matrix.
We evaluate the determinant of \( L_a-\lambda I \) by applying the Schur complement,
\[
\det(L_a - \lambda I) = \det\begin{pmatrix}
a_1 - \lambda & r \\
u & (a_1 - \lambda) I
\end{pmatrix}
= (a_1 - \lambda)^{2n-1} \left( (a_1 - \lambda)^2 - r \cdot u \right),
\]
valid for \( \lambda\neq a_1 \), the block \( (a_1-\lambda)I \) being invertible there. Note that both
sides are polynomials in \( \lambda \), so that the identity holds for every \( \lambda \). Now
\( r \cdot u = -\|u\|^2 \), and \( \dim\V_{2n+1}=2n+1 \) is odd, so that the monic characteristic
polynomial is \( \chi_{L_a}(\lambda)=\det(\lambda I-L_a)=-\det(L_a-\lambda I) \), which is the stated
formula. Part~(2) follows by setting \( \lambda=0 \).
For (3), the matrix of \( R_a \) is obtained from that of \( L_a \) by replacing \( \omega_0 \) by \( -\omega_0 \) in the row \( r \), which changes neither \( r\cdot u=-\|u\|^2 \) nor the block structure, so the computation is unchanged.
\end{proof}

As an immediate application of the above result, we characterize the zero divisors in the Vidinli
algebra \( \V_{2n+1} \). Recall that \( a \) is a \emph{left zero divisor} if \( a\vm b=0 \) for some
\( b\neq0 \), and a \emph{right zero divisor} if \( b\vm a=0 \) for some \( b\neq0 \), equivalently,
\( L_a \), respectively \( R_a \), is singular. Note that the two notions agree in \( \Vn \) by
Proposition~\ref{prop:charpoly}(3), so that we speak of zero divisors without further qualification.
\begin{proposition}\label{prop:zerodivisors}
An element \( 0\neq a \in \V_{2n+1} \) is a zero divisor if and only if \( a \in e_1^\perp \).
\end{proposition}

\begin{proof}
By the determinant formula, \( \det L_a=a_1^{2n-1}\|a\|^{2} \) vanishes if and only if \( a_1=0 \). For \( a\neq0 \) the singularity of \( L_a \) therefore corresponds precisely to the condition \( a\cdot e_1=0 \), that is, to \( a\in e_1^\perp \).
\end{proof}

Note that the zero divisors therefore form a linear subvariety, namely the hyperplane orthogonal to \( e_1 \). 

The non-associative structure of the Vidinli algebra admits a distinguished involution that reflects vectors across the unit direction \( e_1 \). This operation behaves analogously to conjugation in the complex numbers or quaternions, but is made to match the directional asymmetry of the Vidinli multiplication.

\begin{definition}
For any \( a \in \V_{2n+1} \), define its \emph{\( e_1 \)-conjugate} as the reflection of \( a \) across the line \( \mathbb{R} e_1 \). Explicitly
$\bar{a} := 2(a \cdot e_1)\,e_1 - a$.
\end{definition}

\noindent The following are easy consequences of the definition.

\smallskip
\noindent (1) \( \bar{a} = a \) if and only if \( a \in \mathbb{R} e_1 \);

\smallskip
\noindent (2) \( \bar{a} = -a \) if and only if \( a \in V_0 \);

\smallskip
\noindent (3) \( \bar{\bar{a}} = a \) for all \( a \in \V_{2n+1} \), that is, \( \bar{\cdot} \) is an involution;

\smallskip
\noindent (4) \( \| \bar{a} \| = \| a \| \).

\smallskip

\noindent As an application of the conjugation map, we determine the inverses of elements. Since the product is not associative, the two-sided inverse is weaker than one might expect. The solutions of \( a\vm b=b\vm a=e_1 \) always exist for \( a\neq0 \), but they are unique only off the hyperplane \( V_0 \) when \(n > 1\). The next proposition shows the exact situation. In particular it shows that a nonzero element of \( V_0 \) is simultaneously a two-sided zero divisor and possesses a whole affine space of two-sided inverses.

\begin{proposition}
\label{prop:invertibility}
Let \(\V_{2n+1}=(\R^{2n+1},\vm)\) be the Vidinli algebra with unit \(e_1\). For all \(a\in\V_{2n+1}\),
\[
\overline{a}\,\vm\,a \;=\; a\,\vm\,\overline{a}
\;=\;\|a\|^2\,e_1 .
\]
Let \( 0\neq a=a_1e_1+u \) with \( u\in V_0 \). Then 
when $a$ is not a zero divisor, \( \overline a/\|a\|^2 \) is the unique \( b \) with \( a\vm b=b\vm a=e_1 \). If instead \( a\in V_0\setminus\{0\} \), then \( a \) is a two-sided zero divisor and
\[
\{\,b\ :\ a\vm b=b\vm a=e_1\,\}=\{\,v\in V_0\ :\ u\cdot v=-1,\ \omega(u,v)=0\,\}
\]
is an affine subspace of \( V_0 \) of dimension \( 2n-2 \), containing \( \overline a/\|a\|^2=-u/\|u\|^2 \). It reduces to a single point precisely when \( n=1 \).
\end{proposition}

\begin{proof}
We write \(a=a_1e_1+u\) and \(\overline a=a_1e_1-u\), and we let \( b=b_1e_1+v \) with \( v\in V_0 \). By definition,
\begin{equation}\label{eq:prodcoords}
a\vm b=\bigl(a_1b_1-u\cdot v+\omega(u,v)\bigr)e_1+a_1v+b_1u ,
\end{equation}
and \( b\vm a \) is obtained by changing the sign of the \( \omega \)-term. Taking \( b=\overline a \), we get \( \omega(u,-u)=0 \) and \( \overline a\vm a=a\vm\overline a=(a_1^2+\|u\|^2)e_1=\|a\|^2e_1 \).

By \eqref{eq:prodcoords}, the system \( a\vm b=b\vm a=e_1 \) is equivalent to
\[
a_1v+b_1u=0,\qquad \omega(u,v)=0,\qquad a_1b_1-u\cdot v=1 .
\]
If \( a_1\neq0 \), from the first equation, \( v=-(b_1/a_1)u \), whence \( \omega(u,v)=0 \) automatically and the third becomes \( b_1\|a\|^2/a_1=1 \); so \( b_1=a_1/\|a\|^2 \) and \( v=-u/\|a\|^2 \), that is \( b=\overline a/\|a\|^2 \), uniquely. If \( a_1=0 \) and \( u\neq0 \), the first equation forces \( b_1=0 \) and the remaining two read \( u\cdot v=-1 \) and \( \omega(u,v)=(Ju)\cdot v=0 \). The functionals \( v\mapsto u\cdot v \) and \( v\mapsto (Ju)\cdot v \) are linearly independent, because \( u \) and \( Ju \) are orthogonal and nonzero, so the solution set is an affine subspace of \( V_0 \) of dimension \( 2n-2 \); it is a point exactly when \( n=1 \). Finally, \( \overline a/\|a\|^2=-u/\|u\|^2 \) lies in it, as required.
\end{proof}

\begin{example}
In \( \V_5 \) one has \( e_2\vm(-e_2+se_4+te_5)=(-e_2+se_4+te_5)\vm e_2=e_1 \) for all \( s,t\in\R \), so the nonzero element \( e_2\in V_0 \) has a two-parameter family of two-sided inverses, while it is at the same time a two-sided zero divisor: \( e_2\vm(e_2+e_3)=0=(e_2-e_3)\vm e_2 \). 
\end{example}

\section{Subalgebras of the Vidinli algebra}\label{sec:subalgebras}

The subalgebras of \( \Vn \) admit a complete and elementary description as shown below. In the remainder of the section, we identify the isomorphism types that occur in the geometrically significant cases, and the next section shows that the subalgebra geometry determines the Vidinli algebra.

\begin{theorem}\label{thm:sub-lattice}
The assignment \( B\mapsto B\cap V_0 \) is an isomorphism of lattices
\[
\{\text{nonzero subalgebras of }\Vn\}\ \longleftrightarrow\ \{\text{subspaces of }V_0\},
\]
with inverse \( L\mapsto\R e_1\oplus L \). Moreover, for a subspace \( L\subseteq V_0 \) the subalgebra \( \R e_1\oplus L \) is the Vidinli-type algebra built from the restricted data, that is, from the inner product of \( L \) and the form \( \omega|_L \).  
\end{theorem}

\begin{proof}
If \( B \) is a nonzero subalgebra, then \( e_1\in B \) by Proposition~\ref{pro:basics}(3), so \( B=\R e_1\oplus(B\cap V_0) \). Conversely, for any subspace \( L\subseteq V_0 \) the subspace \( \R e_1\oplus L \) is closed under \( \vm \), since \( a\vm b\in\Span_\R\{e_1,a,b\} \) for all \( a,b \). The two assignments are mutually inverse and inclusion-preserving, hence lattice isomorphisms. For the last statement, note that for \( u,v\in L \) one has \( u\vm v=(\omega(u,v)-u\cdot v)e_1 \),
which is precisely the Vidinli-type product built from the inner product of \( L \) and the form
\( \omega|_L \).
\end{proof}

\begin{remark}
Theorem~\ref{thm:type-classification} can be applied to further conclude that
the subalgebra $L$ in the above theorem is commutative if and only if \( L \) is isotropic, and, for \( \dim L=2k \), it is
isomorphic to \( \V_{2k+1} \) if and only if \( L \) is \( J \)-invariant, equivalently if and only if all singular values of \( \omega|_L \) equal \( 1 \). Indeed, by the classification theorem, \( \R e_1\oplus L\cong\V_{2k+1} \) exactly when
all singular values of \( \omega|_L \) equal \( 1 \).) The skew-adjoint operator of \( \omega|_L \)
with respect to the induced inner product is \( P_LJ|_L \), where \( P_L \) denotes the orthogonal
projection onto \( L \), since \( \omega(u,v)=(Ju)\cdot v=(P_LJu)\cdot v \) for \( u,v\in L \). Its
singular values all equal \( 1 \) if and only if \( \|P_LJu\|=\|Ju\|=\|u\| \) for every \( u\in L \),
that is, if and only if \( Ju\in L \) for every \( u\in L \).

\end{remark}

By the theorem, questions about subalgebras are reduced to questions about linear symplectic geometry on \( V_0 \). What it does not describe is how the subalgebras of a fixed isomorphism type sit inside \( \Vn \).
We now determine some cases of interest.  Following the notation for the three-dimensional Vidinli algebra in \cite{CEK}, we call a $k$-dimensional subspace containing $e_1$ a \emph{principal $k$-plane}, and a principal $2$-plane simply a \emph{principal plane}. If \( W \) is a principal \( k \)-plane, then \( W\cap V_0 \) has dimension \( k-1 \), since \( e_1\notin V_0 \), and \( W=\R e_1\oplus(W\cap V_0) \). Choosing an orthonormal basis \( u_1,\dots,u_{k-1} \) of \( W\cap V_0 \) we set \( W=\Span_\R\{e_1,u_1,\dots,u_{k-1}\} \). If $u\in\R^{2n+1}$ is a unit vector not equal to $\pm e_1$, we write $\Pi_u$ for the span of $\{e_1,u\}$.

\begin{proposition}\label{thm:complex}
Let \(\V_{2n+1}\) be the Vidinli algebra with unit \(e_1\) and hyperplane \(V_0=e_1^\perp\). Fix a
unit vector \( u\in\R^{2n+1} \) with \( u\neq\pm e_1 \). Then the principal plane \( \Pi_u \) is
closed under \( \vm \), and \( (\Pi_u,\vm) \) is isometrically isomorphic to \( \C \) by an
isomorphism carrying \( e_1 \) to \( 1 \), where \( \C \) carries the inner product
\( z\cdot w=\operatorname{Re}(\bar zw) \). Equivalently, writing \( u=\alpha e_1+u_0 \) for
\( u_0\in V_0 \), and setting \( a=\lambda_1e_1+\lambda_2u \), \( b=\mu_1e_1+\mu_2u \), then
\begin{equation}\label{eq:vc}
a\vm b
=\bigl(\lambda_1\mu_1 - \lambda_2\mu_2\bigr)e_1
+\bigl(\lambda_1\mu_2 + \lambda_2\mu_1 + 2\alpha\,\lambda_2\mu_2\bigr)u.
\end{equation}
Moreover by changing the basis via
$\psi(e_1)=1,\,\psi(u)=\eta$ one
identifies \(\Pi_u\) with \(\C\). Here \(\eta\) satisfies the quadratic equation \(\eta^2 = -1 + 2\alpha\,\eta\) whose discriminant is \( \Delta = 4(\alpha^2 - 1) < 0 \).
\end{proposition}

\begin{proof}
Let \( u \in \mathbb{R}^{2n+1} \) be a unit vector, and write \( \alpha := u \cdot e_1 \in \mathbb{R} \). Any vector in \( \Pi_{u} \) can be written as \( a = \lambda_1 e_1 + \lambda_2u \) and \( b = \mu_1 e_1 + \mu_2 u \). With this notation, a direct computation gives Equation~\eqref{eq:vc}. Now define a linear map \( \psi: \Pi_{u} \to \mathbb{C} \) by \( \psi(e_1) = 1 \), \( \psi(u) = \eta \in \mathbb{C} \). Then \( \psi(a) = \lambda_1 + \lambda_2 \eta \), and \( \psi(b) = \mu_1 + \mu_2 \eta \), so their product is
$\psi(a)\psi(b) = \lambda_1 \mu_1 + (\lambda_1 \mu_2 + \lambda_2 \mu_1)\eta + \lambda_2 \mu_2 \eta^2$.
In order for this to equal \( \psi(a \vm b) \), we require
$\eta^2 = -1 + 2\alpha \eta$.
Note that \( |\alpha|<1 \), by the Cauchy-Schwarz inequality applied to the unit vectors \( u \) and
\( e_1 \), which are not proportional. Hence the discriminant \( \Delta = 4(\alpha^2 - 1) \) is
negative and \( \eta \in \mathbb{C} \setminus \mathbb{R} \), explicitly \( \eta=\alpha+i\sqrt{1-\alpha^2} \). Thus,
for any \( u \notin \mathbb{R}e_1 \), the map \( \psi \) defines an algebra isomorphism between
\( (\Pi_{u}, \vm) \) and \( \mathbb{C} \). It is moreover an isometry. Indeed since \( |\eta|^2=\alpha^2+(1-\alpha^2)=1 \)
and \( \operatorname{Re}\eta=\alpha=e_1\cdot u \), we have
\[
|\psi(a)|^2=|\lambda_1+\lambda_2\eta|^2=\lambda_1^2+2\lambda_1\lambda_2\alpha+\lambda_2^2=\|\lambda_1e_1+\lambda_2u\|^2 .
\]
\end{proof}

\begin{remark}
When \( u \perp e_1 \), i.e. \( \alpha = 0 \), the formula simplifies to
$a \vm b = (\lambda_1 \mu_1 - \lambda_2 \mu_2)\, e_1 + (\lambda_1 \mu_2 + \lambda_2 \mu_1)\,u$,
which is the standard complex multiplication under the identification \( e_1 \mapsto 1 \), \(u \mapsto i \).
\end{remark}

Next, we consider the \(3\)-dimensional subspaces of \(\R^{2n+1}\) containing the unit \(e_1\). Every such subspace is a subalgebra, but its isomorphism type varies with the geometry of the \(2\)-plane it cuts in \(V_0\). It turns out that the complex structure \( J \) of Section~\ref{sec:defineV} controls it.

\begin{theorem}\label{thm:3dVidinlichar}
Let \( W = \Span_\R\{e_1,u,v\} \) with \( u,v\in V_0 \) orthonormal, and set \( t := \omega(u,v) \). Then \( t\in[-1,1] \) and \( (W,\vm) \) is isomorphic to the twisted Vidinli algebra \( \mathcal{V}_{|t|} \) of Section~\ref{sec:v3}. More precisely

\smallskip
\noindent (1) \( W\cong\V_3 \) if and only if \( |t|=1 \), equivalently \( v=\pm Ju \).

\smallskip
\noindent (2) \( W\cong\mathcal{V}_{|t|} \), not isomorphic to \( \V_3 \), for \( 0<|t|<1 \).

\smallskip
\noindent (3) \( W\cong\VJ_3 \) if and only if \( t=0 \), equivalently \( u \) and \( v \) are \( \omega \)-orthogonal.

\smallskip
\noindent For \( n\ge2 \), as \( u,v \) range over orthonormal pairs in \( V_0 \), the restrictions realize the full family \( \{\mathcal{V}_t\}_{0\le t\le 1} \) inside \( \Vn \). This fails for \( n=1 \): if \( \dim V_0=2 \), then \( \omega(u,v)=\det[u\ v]=\pm1 \) for every orthonormal pair, and only \( t=1 \) occurs.
\end{theorem}

\begin{proof}
Note first that \( |t| \) does not depend on the chosen orthonormal basis of \( W\cap V_0 \). Indeed,
if \( (u',v')=(u,v)R \) with \( R\in\Or(2) \), then \( \omega(u',v')=\det(R)\,\omega(u,v) \) and
\( \det(R)=\pm1 \). Note next that \( \omega(u,v)=(Ju)\cdot v \) with \( \|Ju\|=\|u\|=1 \), so that
the inequality \( |t|\le1 \) holds by the Cauchy-Schwarz inequality. Now let $W$ be as in the theorem. Then it is closed under \( \vm \), since \( w_1\vm w_2\in\Span_\R\{e_1,w_1,w_2\}\subseteq W \). With \( u,v \) orthonormal and \( e_1 \) the unit, the multiplication table is
$u\vm v = t\,e_1,\ v\vm u = -t\,e_1,\ u\vm u = v\vm v = -e_1$,
together with \( e_1\vm w = w \). This coincides with the multiplication table of the twisted Vidinli multiplication~\eqref{eq:twisted} at \( T=t \). Hence \( e_1\mapsto e_1 \), \( u\mapsto e_2 \), \( v\mapsto e_3 \) is an isomorphism \( (W,\vm)\cong\mathcal{V}_t \). The isomorphism types (1)-(3) now follow from the classification in \cite[Theorem~2]{CEK}, together
with the identifications \( \mathcal V_1=\V_3 \) and \( \mathcal V_0=\VJ_3 \). Note that the condition
\( |t|=1 \) in (1) is the case of equality in the Cauchy-Schwarz inequality
\( |(Ju)\cdot v|\le\|Ju\|\,\|v\|=1 \), hence is equivalent to \( v=\pm Ju \). Also the condition \( t=0 \) in (3) implies exactly that \( u \) and \( v \) are \( \omega \)-orthogonal.

Finally, let \( n\ge2 \) and fix a unit vector \( u\in V_0 \). We set \( U:=u^\perp\cap V_0 \), so that \( \dim U=2n-1 \) and \( Ju\in U \). As \( v \) ranges over the unit sphere \( S \) of \( U \), the value \( t=(Ju)\cdot v \) depends continuously on \( v \), and it equals \( 1 \) at \( v=Ju \). Moreover \( \dim\bigl(U\cap(Ju)^\perp\bigr)=2n-2\ge2 \), so that \( S \) contains a vector orthogonal to \( Ju \), at which \( t=0 \). Since \( \dim S=2n-2\ge2 \), the sphere \( S \) is connected, and therefore every value \( t\in[0,1] \) is attained. For \( n=1 \) the sphere \( S \) reduces to the pair \( \{\pm Ju\} \), and only \( |t|=1 \) occurs, as required.
\end{proof}

In particular, by Part~(1), the unique principal \( 3 \)-plane through a given unit \( u\in V_0 \) isomorphic to \( \V_3 \) is \( \Span_\R\{e_1,u,Ju\} \).

\begin{corollary}
\label{cor:canonical-v3}
With the above conventions, any standard basis vector \( u = e_{2i} \in V_0 \) has a canonical symplectic dual \( v := e_{2i+1} \) satisfying
$u \cdot v = 0,\ \omega(u,v) = 1, \ \|u\| = \|v\| = 1$.
Hence the subspace 
$W_i := \Span_\R\{e_1, e_{2i}, e_{2i+1}\} \;\subset\; \mathbb{V}_{2n+1}$
is an explicit copy of the 3-dimensional Vidinli algebra \( \mathbb{V}_3 \).
\end{corollary}

Although \( \Vn \) is neither commutative nor associative, it contains a rich class of commutative
subalgebras, and we determine them next. Following \cite[Section 1.2]{daSilva}, a subspace \( L \) of
\( V_0 \) is \emph{isotropic} if \( \omega|_L=0 \), and \emph{Lagrangian} if in addition
\( \dim L=\tfrac12\dim V_0 \). We call a commutative subalgebra of \( \Vn \) a
\emph{Vidinli-Jordan subalgebra}; the product then restricts to the symmetric operation \( \vijo \),
so it is a Vidinli-Jordan algebra in the sense of Section~\ref{sec:ViJo}.

\begin{theorem}\label{thm:LGrass}
A subalgebra \( B\subseteq\Vn \) is commutative if and only if \( B=\R e_1\oplus L \) with
\( L\subseteq V_0 \) isotropic, and then \( B\cong\VJ_{\dim L+1} \). Such a \( B \) has dimension at
most \( n+1 \), with equality precisely when \( L \) is Lagrangian. Hence \( L\mapsto\R e_1\oplus L \)
is a bijection between the Lagrangian subspaces of \( V_0 \) and the maximal Vidinli-Jordan
subalgebras of \( \Vn \), all of which have dimension \( n+1 \).
\end{theorem}

\begin{proof}
By Theorem~\ref{thm:sub-lattice}, \( B=\R e_1\oplus L \) with \( L=B\cap V_0 \). For \( u,v\in L \) the commutator is \( 2\,\omega(u,v)e_1 \). Hence \( B \) is commutative if and only if \( \omega|_L=0 \), in which case \( u\vm v=-(u\cdot v)e_1 \), which is the Vidinli-Jordan product, so that \( B\cong\VJ_{\dim L+1} \). An isotropic subspace of a \( 2n \)-dimensional symplectic space has dimension at most \( n \), with equality exactly for the Lagrangian subspaces. Since \( L\mapsto\R e_1\oplus L \) is inclusion-preserving by Theorem~\ref{thm:sub-lattice}, maximal commutative subalgebras correspond to maximal isotropic subspaces, as required.
\end{proof}

\section{The local-to-global characterization}
\label{sec:axiomatic-classification}

The subalgebra geometry of the previous two sections is rigid enough to reconstruct the multiplication.
We show in this section that a unital product on \( \R^{2n+1} \) is the Vidinli multiplication as soon
as one knows two things: which of its principal planes are copies of \( \C \), and which of its
principal \( 3 \)-planes are copies of \( \V_3 \). Note that neither the symplectic form nor any
numerical normalization is required.

\begin{theorem}\label{thm:vidinli-characterization}
Let \( n\ge 1 \) and let \( A=(\R^{2n+1},\ast) \) be a unital \( \R \)-algebra, where
\( \R^{2n+1} \) carries its standard inner product and the unit of \( A \) is the basis vector
\( e_1 \). Consider the conditions
\begin{itemize}
\item[$\mathrm{(C)}$] every principal \( 2 \)-plane of \( A \) is a subalgebra, isometrically
isomorphic to \( \C \) by a map sending \( e_1 \) to \( 1 \) where \( \C \) is given the inner product \( z\cdot w=\operatorname{Re}(\bar zw) \);
\item[$\mathrm{(S)}$] every principal \( 3 \)-plane of \( A \) is a subalgebra;
\item[$\mathrm{(U)}$] through every unit vector of \( V_0 \) there passes at least one, and only
finitely many, principal \( 3 \)-planes isomorphic to \( \V_3 \).
\end{itemize}
Then \( A \) satisfies $\mathrm{(C)}$, $\mathrm{(S)}$ and $\mathrm{(U)}$ if and only if \( A \) is
isometrically isomorphic to \( \V_{2n+1} \).
\end{theorem}

\begin{remark}\label{rem:characterization}
\noindent (a) Taking the unit to be \( e_1 \) is a normalization and not a restriction. Condition
\( \mathrm{(C)} \) forces \( \|e_1\|=1 \), and the three conditions are phrased through the standard
inner product alone, hence are \( \Or(2n+1) \)-equivariant; so any unital algebra structure on
Euclidean \( \R^{2n+1} \) whose unit is a unit vector can be rotated into the position assumed here.

\smallskip
\noindent (b) Given that \( e_1 \) is the unit, condition \( \mathrm{(C)} \) is equivalent to the
single quadratic condition \( x\ast x=-\|x\|^2e_1 \) for all \( x\in V_0 \), that is, to \( A \)
being a conic algebra whose norm is the Euclidean one. One implication is the first paragraph of the proof of Lemma~\ref{lem:symmetric} below. Conversely, if \( x\ast x=-\|x\|^2e_1 \) on \( V_0 \), then for a
unit vector \( u\in V_0 \) the plane \( \R e_1\oplus\R u \) is closed under \( \ast \) and
\( e_1\mapsto1 \), \( u\mapsto i \) is an isometric isomorphism onto \( \C \); and every principal
\( 2 \)-plane is of this form. So \( \mathrm{(C)} \) fixes the quadratic form and nothing else.

\smallskip
\noindent (c) The proof below shows that the number of principal \( 3 \)-planes isomorphic to
\( \V_3 \) through a given unit vector of \( V_0 \) is \( 0 \), \( 1 \) or infinite. In the presence of (C) and (S), condition~(U) is therefore equivalent to requiring
exactly one such plane through each unit vector of $V_0$, for every $n\geq1$.
\smallskip
\noindent (d) For \( n=1 \) condition \( \mathrm{(S)} \) is vacuous and \( A \) itself is the only
principal \( 3 \)-plane, so that \( \mathrm{(U)} \) says exactly that \( A\cong\V_3 \). In that case
the theorem follows from Lemmas~\ref{lem:symmetric} and \ref{lem:V3crit} alone: the symmetric part of
\( \ast \) is the standard one, and \( K(u,v)=\pm e_1 \) for an orthonormal pair \( u,v\in V_0 \), so
that \( K_0=0 \) and \( \ast \) is the Vidinli-type product attached to \( \pm\omega_0 \); the
reflection \( e_3\mapsto-e_3 \) is then an isometric isomorphism onto \( \V_3 \). The proof below
uses Lemma~\ref{lem:shape} and the case analysis, and is written for \( n\ge2 \).
\end{remark}

The hypotheses are of the intended kind. Condition $\mathrm{(C)}$ records which principal \( 2 \)-planes are copies of \( \C \), namely all of them, and $\mathrm{(U)}$ records which principal \( 3 \)-planes are copies of \( \V_3 \), namely one through each direction. Condition $\mathrm{(S)}$ is not about isomorphism types at all. It asks only that the principal \( 3 \)-planes be closed under the product. The proof of this theorem is in several steps and we separate some of the steps as lemmas. The first one determines the symmetric part, the second realizes $\mathrm{(U)}$ as a statement about the skew part, and the third lemma determines the skew part up to two parameters. For the rest we assume the notation of the above theorem.

\begin{lemma}\label{lem:symmetric}
Assume $\mathrm{(C)}$ and write \( \ast=S+K \) with \( S \) symmetric and \( K \) skew-symmetric. Then
\[
S(a,b)=(a\cdot e_1)\,b+(b\cdot e_1)\,a-(a\cdot b)\,e_1 ,
\qquad K(e_1,\cdot)=0 .
\]
In particular \( x\ast x=-\|x\|^2e_1 \) for all \( x\in V_0 \), and \( u\ast v=K(u,v) \) for all orthonormal pairs of vectors
\( u,v\in V_0 \).
\end{lemma}

\begin{proof}
For a unit \( u\in V_0 \) let \( \varphi\colon\R e_1\oplus\R u\to\C \) be an isometric isomorphism with
\( \varphi(e_1)=1 \), and put \( \eta:=\varphi(u) \). Then \( |\eta|=\|u\|=1 \) and
\( \operatorname{Re}\eta=1\cdot\eta=e_1\cdot u=0 \), so \( \eta=\pm i \) and
\( u\ast u=\varphi^{-1}(\eta^2)=-e_1 \). By homogeneity, \( x\ast x=-\|x\|^2e_1 \) for every
\( x\in V_0 \). Hence, for \( a=\alpha e_1+x \),
\[
S(a,a)=a\ast a=\alpha^2e_1+2\alpha x-\|x\|^2e_1 .
\]
Here we have used that \( e_1 \) is the unit. Note that the right-hand side is the value at
\( (a,a) \) of the symmetric bilinear map
\( T(a,b)=(a\cdot e_1)b+(b\cdot e_1)a-(a\cdot b)e_1 \). Therefore the symmetric bilinear map
\( S-T \) vanishes on the diagonal, and hence identically, by polarization. Finally
\( K(e_1,b)=e_1\ast b-S(e_1,b)=b-b=0 \) for every \( b\in A \).
\end{proof}

\begin{lemma}\label{lem:V3crit}
Assume $\mathrm{(C)}$ and let \( u,v\in V_0 \) be orthonormal. Then
\( \R e_1\oplus\R u\oplus\R v\cong\V_3 \) if and only if \( K(u,v)=\pm e_1 \).
\end{lemma}

\begin{proof}
If \( K(u,v)=\varepsilon e_1 \) with \( \varepsilon=\pm1 \), then by Lemma~\ref{lem:symmetric} the
subspace \( W=\R e_1\oplus\R u\oplus\R v \) is closed under the product. Its multiplication table is \( u\ast u=v\ast v=-e_1 \),
\( u\ast v=\varepsilon e_1=-\,v\ast u \), with unit \( e_1 \). Thus sending \( e_1\mapsto e_1 \),
\( u\mapsto e_2 \), \( \varepsilon v\mapsto e_3 \) extends to an isomorphism onto \( \V_3 \).

Conversely let \( \Phi\colon W\to\V_3 \) be an algebra isomorphism. It maps the unit to the unit, so
\( \Phi(e_1)=e_1 \). For any \( 0\neq x\in W\cap V_0 \) we have \( x\ast x=-\|x\|^2e_1 \) by
Lemma~\ref{lem:symmetric}. Hence \( \Phi(x)\vm\Phi(x)=-\|x\|^2e_1\in\R_{\le0}e_1 \). In particular by
Proposition~\ref{pro:basics}(2), applied to \( \V_3 \), we get \( \Phi(x)\in e_1^\perp \) with
\( \|\Phi(x)\|=\|x\| \). Thus \( \Phi \) restricts to a linear isometry of
\( W\cap V_0=\Span_\R\{u,v\} \) onto \( e_1^\perp\subset\V_3 \). Thus \( (\Phi u,\Phi v) \) is an
orthonormal basis of \( \R^2 \) and we have
\[
\Phi\bigl(K(u,v)\bigr)=\Phi(u\ast v)=\Phi(u)\vm\Phi(v)=\omega_0(\Phi u,\Phi v)\,e_1
=\pm e_1 .
\]
Since \( \Phi(e_1)=e_1 \) and \( \Phi \) is injective, \( K(u,v)=\pm e_1 \).
\end{proof}

\begin{lemma}\label{lem:shape}
Let \( W \) be a real vector space with \( \dim W\ge3 \) and let \( F\colon W\times W\to W \) be a
skew-symmetric bilinear map with \( F(u,v)\in\Span_\R\{u,v\} \) for all \( u,v\in W \). Then
\( F(u,v)=\mu(u)\,v-\mu(v)\,u \) for a unique linear functional \( \mu\in W^{*} \).
\end{lemma}

\begin{proof}
Fix \( 0\neq u\in W \) and a complement \( H \) of \( \R u \). We compose \( v\mapsto F(u,v) \) with the
projection of \( W \) onto \( H \) along \( \R u \) to obtain a linear map \( Q_u\colon H\to H \) with
\( Q_u(v)\in\R v \) for every \( v \). Hence we must have \( Q_u=\mu(u)\,\mathrm{id}_H \) for a scalar
\( \mu(u) \), and therefore \( F(u,v)=\beta_u(v)\,u+\mu(u)\,v \) for every \( v\in H \) and some
scalar \( \beta_u(v) \).
Reading the coefficient of \( u \) in \( F(v,u)=-F(u,v) \), and using the same expression by exchanging \( u \) and \( v \), we get \( \mu(v)=-\beta_u(v) \). Therefore \( F(u,v)=\mu(u)v-\mu(v)u \) for all independent
\( u,v \). In particular \( \mu(u) \) does not depend on the choice of the complement \( H \): it is the
coefficient of \( v \) in \( F(u,v) \) for any \( v \) independent from \( u \). Moreover both sides
of the identity \( F(u,v)=\mu(u)v-\mu(v)u \) are linear in \( v \) and vanish at \( v=u \), so that
once it holds for \( v\in H \) it holds for every \( v\in W=H\oplus\R u \).

To prove that \( \mu \) is linear, note that its homogeneity follows from \( F(tu,v)=tF(u,v) \). Note
also that the formula then holds for dependent \( u,v \) as well, both sides vanishing. If \( u,u' \) are
independent, choose \( v\notin\Span_\R\{u,u'\} \) which is possible since \( \dim W\ge3 \). Now comparing the
coefficient of \( v \) in \( F(u+u',v)=F(u,v)+F(u',v) \), we obtain the equation \( \mu(u+u')=\mu(u)+\mu(u') \). The case of
dependent vectors follows by homogeneity. Finally uniqueness is clear, since \( \mu(u) \) is recovered as the
coefficient of \( v \) in \( F(u,v) \) for any \( v \) independent of \( u \).
\end{proof}

\begin{proof}[Proof of Theorem~\ref{thm:vidinli-characterization}]
Assume $\mathrm{(C)}$, $\mathrm{(S)}$ and $\mathrm{(U)}$. By Lemma~\ref{lem:symmetric} the symmetric
part of \( \ast \) is the standard one, so it remains to determine \( K \) on \( V_0\times V_0 \).

We write \( K=K_1e_1+K_0 \), where \( K_1 \) is a scalar-valued skew form on \( V_0 \) and
\( K_0\colon V_0\times V_0\to V_0 \) is a skew-symmetric bilinear map. Condition $\mathrm{(S)}$ says exactly that
\( K(u,v)\in\Span_\R\{e_1,u,v\} \), that is, \( K_0(u,v)\in\Span_\R\{u,v\} \). As
\( \dim V_0=2n\ge4 \), we get \( \mu\in V_0^{*} \) with
\( K_0(u,v)=\mu(u)v-\mu(v)u \) by  Lemma~\ref{lem:shape}. Now let \( u\in V_0 \) be any unit vector. By $\mathrm{(U)}$ there is a principal \( 3 \)-plane
\( \R e_1\oplus L \) with \( u\in L \) and \( \R e_1\oplus L\cong\V_3 \). Taking for \( v \) a unit
vector of \( L\cap u^\perp \), we obtain an orthonormal pair \( u,v \) with
\( \R e_1\oplus\R u\oplus\R v\cong\V_3 \). Thus \( K(u,v)=\pm e_1 \) by
Lemma~\ref{lem:V3crit}. Comparing \( V_0 \)-components, we obtain \( \mu(u)v-\mu(v)u=0 \), and since \( u,v \) are
independent, we must have \( \mu(u)=0 \). As \( u \) was arbitrary, \( \mu=0 \) and hence \( K_0=0 \). Setting
\( \omega:=K_1 \), extended by \( \omega(e_1,\cdot)=0 \), we obtain
\[
a\ast b=(a\cdot e_1)\,b+(b\cdot e_1)\,a-(a\cdot b)\,e_1+\omega(a,b)\,e_1 ,
\]
with \( \omega \) alternating. Let \( A_\omega\in\End(V_0) \) be its skew-adjoint operator,
\( \omega(u,v)=(A_\omega u)\cdot v \), so that \( (A_\omega u)\cdot u=0 \).

Fix a unit vector \( u\in V_0 \) and put \( \rho:=\|A_\omega u\| \). By Lemma~\ref{lem:V3crit} the principal
\( 3 \)-planes through \( u \) isomorphic to \( \V_3 \) are exactly the
\( \R e_1\oplus\R u\oplus\R v \) with \( v \) a unit vector of \( U:=u^\perp\cap V_0 \) satisfying
\( |(A_\omega u)\cdot v|=1 \). Two such \( v \) define the same plane precisely when they differ by a sign. Since \( A_\omega u\in U \) and \( \dim U=2n-1 \), we have the following cases.

\smallskip

i) If \( \rho<1 \) there is no such \( v \), by the Cauchy-Schwarz inequality.

\smallskip

ii) If \( \rho=1 \) equality in Cauchy-Schwarz implies \( v=\pm A_\omega u \), hence we get exactly one plane.

\smallskip

iii) If \( \rho>1 \) the solutions of \( (A_\omega u)\cdot v=1 \) on the unit sphere of \( U \)
form a sphere of dimension \( 2n-3\ge1 \), hence there are infinitely many planes.

\smallskip

Condition $\mathrm{(U)}$ therefore forces \( \rho=1 \): case i) is excluded because at least one such
plane is required, and case iii) because only finitely many are allowed. That is,
\( \|A_\omega u\|=\|u\| \) for every unit \( u \), hence for every \( u\in V_0 \) by homogeneity. Polarizing
\( (A_\omega x)\cdot(A_\omega x)=x\cdot x \) one gets
\( (A_\omega x)\cdot(A_\omega y)=x\cdot y \), so \( A_\omega \) is orthogonal; combined
with \( A_\omega^{\mathsf T}=-A_\omega \) this yields
\( \mathrm{id}=A_\omega^{\mathsf T}A_\omega=-A_\omega^2 \), that is
\( A_\omega^2=-\,\mathrm{id}_{V_0} \). In particular \( \omega \) is nondegenerate.

We construct an orthonormal symplectic basis by induction on \( n \). Choose a unit \( f_2\in V_0 \) and
set \( f_3:=A_\omega f_2 \). Then \( \|f_3\|=1 \), \( f_2\cdot f_3=(A_\omega f_2)\cdot f_2=0 \)
and \( \omega(f_2,f_3)=(A_\omega f_2)\cdot(A_\omega f_2)=1 \). The plane
\( P=\Span_\R\{f_2,f_3\} \) is \( A_\omega \)-invariant because \( A_\omega^2=-\mathrm{id} \), and so is
\( P^\perp \) because \( A_\omega \) is orthogonal; the restriction of \( A_\omega \) to \( P^\perp \)
satisfies the same hypotheses, so the induction applies. In the resulting basis
\( f_2,\dots,f_{2n+1} \) we have \( \omega(f_{2i},f_{2i+1})=1 \) with all other pairings of basis
vectors zero, so \( \omega=\omega_0 \) and \( \ast \) is the multiplication of
Definition~\ref{defn:Vidinli}; alternatively, Proposition~\ref{prop:rigidity} applies with
\( \delta_1=\cdots=\delta_n=1 \). The basis \( e_1,f_2,\dots,f_{2n+1} \) being orthonormal, the linear map
carrying it to the standard basis lies in \( \Or(2n+1) \) and is an isomorphism \( A\to\V_{2n+1} \).

Conversely, \( \V_{2n+1} \) satisfies all three conditions. Note that every principal \( 2 \)-plane is
\( \Pi_u \) for some unit vector \( u\in V_0 \), that \( u\vm u=-e_1 \) by
Proposition~\ref{pro:basics}(2), and that \( e_1\mapsto1 \), \( u\mapsto i \) then defines an
isometric isomorphism onto \( \C \); hence $\mathrm{(C)}$ holds. Condition $\mathrm{(S)}$ holds since
\( a\vm b\in\Span_\R\{e_1,a,b\} \) by Theorem~\ref{thm:sub-lattice}. Finally, $\mathrm{(U)}$ holds by the counting above: since \( \|Ju\|=1 \) for every unit \( u \),
through each such \( u \) there passes exactly one principal \( 3 \)-plane isomorphic to \( \V_3 \),
namely \( \Span_\R\{e_1,u,Ju\} \), as in Theorem~\ref{thm:3dVidinlichar}(1). All three conditions are
phrased through the standard inner product, so they are invariant under isometric isomorphism. This
completes the proof of the theorem.
\end{proof}

\appendix

\section{A classification theorem}

The following classification is a consequence of the description of Vidinli algebras by bilinear forms in~\cite[Theorem~3.1]{ElduqueRandez}.

We assume the notation of Section 4. It turns out that the right invariant for the classification of Vidinli-type algebras is
metric, and together with the normal form of a skew-adjoint operator it classifies these algebras
completely. For an alternating form \( \tilde\omega \) on \( V_0 \) let
\( A_{\tilde\omega}\in\End(V_0) \) be the associated skew-adjoint operator,
\( \tilde\omega(u,v)=(A_{\tilde\omega}u)\cdot v \).

\begin{theorem}\label{thm:type-classification}
Let \( \tilde\omega,\tilde\omega' \) be alternating forms on \( V_0 \), with skew-adjoint operators
\( A,A' \). Then:
\begin{enumerate}
\item \( \Vn^{\tilde\omega}\cong\Vn^{\tilde\omega'} \) if and only if \( A \) and \( A' \) are
conjugate by an element of \( \Or(V_0) \), if and only if they have the same singular values with the
same multiplicities. Every such isomorphism is automatically an isometry fixing \( e_1 \).
\item There is an orthonormal basis of \( V_0 \) in which \( A \) is block diagonal with \( k \)
blocks \( \left(\begin{smallmatrix}0&t_i\\-t_i&0\end{smallmatrix}\right) \),
\( t_1\ge\cdots\ge t_k>0 \), and a zero block of size \( 2n-2k \). The multiset
\( \{t_1,\dots,t_n\} \), zeros allowed, is a complete isomorphism invariant; \( \Vn \) is the case
\( t_1=\cdots=t_n=1 \) and \( \VJ_{2n+1} \) the case \( t_1=\cdots=t_n=0 \). The form
\( \tilde\omega \) is nondegenerate, that is \( \Vn^{\tilde\omega} \) is of symplectic type, exactly
when \( t_n>0 \).
\item \( \Aut\bigl(\Vn^{\tilde\omega}\bigr) \) is the centralizer of \( A \) in \( \Or(V_0) \),
namely \( \Or(m_0)\times\prod_j \U(m_j) \), where \( m_0=\dim\ker A \) and \( m_j \) is the
multiplicity of the \( j \)-th distinct positive value among the \( t_i \).
\end{enumerate}
\end{theorem}
\begin{proof}
Let \( \Phi\colon\Vn^{\tilde\omega}\to\Vn^{\tilde\omega'} \) be an isomorphism. By
Remark~\ref{rem:type-squares} it fixes \( e_1 \), maps \( V_0 \) onto \( V_0 \) and preserves the
norm there, so \( \psi:=\Phi|_{V_0}\in \Or(V_0) \) and \( \Phi \) is an isometry. Comparing skew
parts of \( \Phi(u\vm^{\tilde\omega}v)=\Phi(u)\vm^{\tilde\omega'}\Phi(v) \) for \( u,v\in V_0 \)
gives \( \tilde\omega(u,v)=\tilde\omega'(\psi u,\psi v) \), that is
\( (Au)\cdot v=(\psi^{\mathsf T}A'\psi u)\cdot v \), so \( A'=\psi A\psi^{-1} \). Conversely any such
\( \psi \), extended by \( e_1\mapsto e_1 \), is an isomorphism, since by
Definition~\ref{defn:Vidinli} the product depends only on the inner product and on the form. Part (2)
is the real normal form of a skew-adjoint operator on a Euclidean space, and orthogonally conjugate
operators have the same singular values; conversely equal singular values give the same normal form,
which proves the second equivalence in (1). Nondegeneracy of \( \tilde\omega \) is invertibility of
\( A \), that is \( t_n>0 \). For (3), \( \psi\in \Or(V_0) \) extends to an automorphism exactly when
\( \psi A=A\psi \). The decomposition of \( V_0 \) into \( \ker A \) and the eigenspaces \( E_j \) of
\( -A^2 \) belonging to the distinct values \( t^2>0 \) is orthogonal and preserved by any such
\( \psi \); on \( \ker A \) the condition is vacuous, giving \( \Or(m_0) \), while on \( E_j \) the
operator \( A/t_j \) is an orthogonal complex structure, so the isometries commuting with it form
\( \U(m_j) \).
\end{proof}

\begin{corollary}\label{cor:normalization}
Let \( \tilde\omega \) be an alternating form on \( V_0 \). Then \( \Vn^{\tilde\omega}\cong\Vn \) if
and only if
\[
\max\bigl\{\,\|u\vm^{\tilde\omega}v\| \;:\; v\in V_0,\ \|v\|=1,\ v\perp u\,\bigr\}=1
\qquad\text{for every unit vector }u\in V_0 .
\]
\end{corollary}
\begin{proof}
For orthonormal \( u,v\in V_0 \) one has
\( u\vm^{\tilde\omega}v=\tilde\omega(u,v)\,e_1=\bigl((A_{\tilde\omega}u)\cdot v\bigr)e_1 \). Since
\( (A_{\tilde\omega}u)\cdot u=\tilde\omega(u,u)=0 \), the vector \( A_{\tilde\omega}u \) is
orthogonal to \( u \), so the maximum of \( |(A_{\tilde\omega}u)\cdot v| \) over unit vectors
\( v\in V_0 \) orthogonal to \( u \) is attained and equals \( \|A_{\tilde\omega}u\| \). The
displayed condition therefore says that \( A_{\tilde\omega} \) is an isometry, that is, that all the
invariants \( t_i \) of Theorem~\ref{thm:type-classification} equal \( 1 \).
\end{proof}

\begin{example}\label{ex:nonisoV5}
On \( \R^5 \), let \( \tilde\omega \) be given by
\( \tilde\omega(e_2,e_3)=\tilde\omega(e_3,e_4)=\tilde\omega(e_4,e_5)=1 \), all other pairings of
basis vectors being zero; \( \tilde\omega \) is nondegenerate. The operator
\( -A_{\tilde\omega}^{2} \) preserves the orthogonal splitting
\( V_0=\Span_\R\{e_2,e_4\}\oplus\Span_\R\{e_3,e_5\} \), acting on the two summands by the matrices
\[
\begin{pmatrix}1&-1\\-1&2\end{pmatrix}
\qquad\text{and}\qquad
\begin{pmatrix}2&-1\\-1&1\end{pmatrix} ,
\]
both with characteristic polynomial \( \lambda^{2}-3\lambda+1 \), whose roots are
\( \varphi^{\pm2} \), where \( \varphi \) denotes the golden ratio. The invariants of
Theorem~\ref{thm:type-classification} are therefore \( t_1=\varphi \) and \( t_2=\varphi^{-1} \),
whereas those of \( \omega_0 \) both equal \( 1 \). Hence \( \V_5^{\tilde\omega}\not\cong\V_5 \),
even though \( \tilde\omega \) and \( \omega_0 \) are symplectomorphic. By
Theorem~\ref{thm:type-classification}(3) its automorphism group is the torus
\( \U(1)\times\U(1) \), rather than \( \U(2) \).
\end{example}

\section*{Acknowledgements}
The coordinate-free expression for the algebra $\V_3$ of Proposition \ref{prop:coordfree3} was proved independently by Yusuf \"Unl\"u. We thank him for sharing his formula. We also thank T\"ur\-ker B\i y\i k\-o\u{g}\-lu, Arzu Boysal, \.Izzet Co\c{s}kun and Gizem Karaali for their valuable comments.
We have used Claude AI to check the proofs and typos. We were suggested Theorem \ref{thm:type-classification} and its proof together with some other small suggestions on the writing. 

\def\elink#1{}

\end{document}